\documentclass[12pt]{amsart}
\usepackage{amsmath,amsthm,amsfonts,amssymb,mathrsfs}
\usepackage{color}

\usepackage{tikz}

\usepackage{amssymb,cite}
\usepackage[colorlinks,plainpages,citecolor=magenta, linkcolor=blue, backref]{hyperref}%,bookmarksnumbered,

\usepackage{hyperref}
\date{\today}

\usepackage{hyperref}
\input{xy}
 \xyoption{all}
 \xyoption{arc}

\newtheorem{theorem}{Theorem}

\newtheorem{proposition}{Proposition}%[section]
\newtheorem{lemma}{Lemma}%[section]
\theoremstyle{definition}

\newtheorem{example}{Example}%[section]
\newtheorem{remark}{Remark}%[section]

\begin{document}

\title[On the semigroup of injective monoid endomor\-phisms of the monoid $\boldsymbol{B}_{\omega}^{\mathscr{F}}$]{On the semigroup of injective monoid endomor\-phisms of the monoid $\boldsymbol{B}_{\omega}^{\mathscr{F}}$ with the two-elements family $\mathscr{F}$ of inductive nonempty subsets of $\omega$}
\author{Oleg Gutik and Inna Pozdniakova}
\address{Ivan Franko National University of Lviv, Universytetska 1, Lviv, 79000, Ukraine}
\email{oleg.gutik@lnu.edu.ua, pozdnyakova.inna@gmail.com}



\begin{abstract}
We study injective endomorphisms of the semigroup $\boldsymbol{B}_{\omega}^{\mathscr{F}}$ with the two-elements family $\mathscr{F}$ of inductive nonempty subsets of $\omega$. We describe the elements of the semigroup $\boldsymbol{End}^1_*(\boldsymbol{B}_{\omega}^{\mathscr{F}})$ of all injective monoid endomorphisms of the monoid $\boldsymbol{B}_{\omega}^{\mathscr{F}}$, and show that Green's relations $\mathscr{R}$, $\mathscr{L}$, $\mathscr{H}$, $\mathscr{D}$, and $\mathscr{J}$  on $\boldsymbol{End}^1_*(\boldsymbol{B}_{\omega}^{\mathscr{F}})$ coincide with the relation of equality.
\end{abstract}

\maketitle

%\tableofcontents

\section{\textbf{Introduction, motivation and main definitions}}

We shall follow the terminology of~\cite{Clifford-Preston-1961, Clifford-Preston-1967, Lawson=1998}. By $\omega$ we denote the set of all non-negative integers, by $\mathbb{N}$ the set of all positive integers, and by $\mathbb{Z}$ the set of all integers.

\smallskip

Let $\mathscr{P}(\omega)$ be  the family of all subsets of $\omega$. For any $F\in\mathscr{P}(\omega)$ and $n\in\mathbb{Z}$ we put $nF=\{nk\colon k\in F\}$ if $F\neq\varnothing$ and $n\varnothing=\varnothing$. A subfamily $\mathscr{F}\subseteq\mathscr{P}(\omega)$ is called \emph{${\omega}$-closed} if $F_1\cap(-n+F_2)\in\mathscr{F}$ for all $n\in\omega$ and $F_1,F_2\in\mathscr{F}$. For any $a\in\omega$ we denote
\begin{equation*}
[a)=\{x\in\omega\colon x\geqslant a\}.
\end{equation*}

A subset $A$ of $\omega$ is said to be \emph{inductive}, if $i\in A$ implies $i+1\in A$. In particular the empty set $\varnothing$ is an inductive subset of $\omega$.

\begin{remark}[\cite{Gutik-Mykhalenych=2021}]\label{remark-1.1}
\begin{enumerate}
  \item\label{remark-1.1(1)} By Lemma~6 from \cite{Gutik-Mykhalenych=2020}, a nonempty subset $F\subseteq \omega$ is inductive in $\omega$ if and only $(-1+F)\cap F=F$.
  \item\label{remark-1.1(2)} Since the set $\omega$ with the usual order is well-ordered, for any nonempty inductive subset $F$ in $\omega$ there exists nonnegative integer $n_F\in\omega$ such that $[n_F)=F$.
  \item\label{remark-1.1(3)} Statement \eqref{remark-1.1(2)} implies that the intersection of an arbitrary finite family of nonempty inductive subsets in $\omega$ is a nonempty inductive subset of  $\omega$.
\end{enumerate}
\end{remark}

A semigroup $S$ is called {\it inverse} if for any
element $x\in S$ there exists a unique $x^{-1}\in S$ such that
$xx^{-1}x=x$ and $x^{-1}xx^{-1}=x^{-1}$. The element $x^{-1}$ is
called the {\it inverse of} $x\in S$. If $S$ is an inverse
semigroup, then the function $\operatorname{inv}\colon S\to S$
which assigns to every element $x$ of $S$ its inverse element
$x^{-1}$ is called the {\it inversion}.

\smallskip

If $S$ is a semigroup, then we shall denote the subset of all
idempotents in $S$ by $E(S)$. If $S$ is an inverse semigroup, then
$E(S)$ is closed under multiplication and we shall refer to $E(S)$ as a
\emph{band} (or the \emph{band of} $S$). Then the semigroup
operation on $S$ determines the following partial order $\preccurlyeq$
on $E(S)$:
\begin{equation*}
e\preccurlyeq f \quad \hbox{ if and only if } \quad ef=fe=e.
\end{equation*}
This order is
called the {\em natural partial order} on $E(S)$. A \emph{semilattice} is a commutative semigroup of idempotents.

\smallskip

If $S$ is an inverse semigroup then the semigroup operation on $S$ determines the following partial order $\preccurlyeq$
on $S$:
\begin{equation*}
s\preccurlyeq t \quad \hbox{ if and only if there exists } \quad e\in E(S) \quad \hbox{ such that } \quad s=te.
\end{equation*}
This order is
called the {\em natural partial order} on $S$ \cite{Wagner-1952}.

\smallskip

The bicyclic monoid ${\mathscr{C}}(p,q)$ is the semigroup with the identity $1$ generated by two elements $p$ and $q$ subjected only to the condition $pq=1$. The semigroup operation on ${\mathscr{C}}(p,q)$ is determined as
follows:
\begin{equation*}
    q^kp^l\cdot q^mp^n=q^{k+m-\min\{l,m\}}p^{l+n-\min\{l,m\}}.
\end{equation*}
It is well known that the bicyclic monoid ${\mathscr{C}}(p,q)$ is a bisimple (and hence simple) combinatorial $E$-unitary inverse semigroup and every non-trivial congruence on ${\mathscr{C}}(p,q)$ is a group congruence \cite{Clifford-Preston-1961}.

\smallskip

On the set $\boldsymbol{B}_{\omega}=\omega\times\omega$ we define the semigroup operation ``$\cdot$'' in the following way
\begin{equation}\label{eq-1.1}
  (i_1,j_1)\cdot(i_2,j_2)=
  \left\{
    \begin{array}{ll}
      (i_1-j_1+i_2,j_2), & \hbox{if~} j_1\leqslant i_2;\\
      (i_1,j_1-i_2+j_2), & \hbox{if~} j_1\geqslant i_2.
    \end{array}
  \right.
\end{equation}
It is well known that the bicyclic monoid $\mathscr{C}(p,q)$ is isomorphic to the semigroup $\boldsymbol{B}_{\omega}$ and the mapping
\begin{equation*}
\mathfrak{h}\colon \mathscr{C}(p,q)\to \boldsymbol{B}_{\omega}, \quad q^kp^l\mapsto (k,l)
\end{equation*}
is an isomorphism
(see: \cite[Section~1.12]{Clifford-Preston-1961} or \cite[Exercise IV.1.11$(ii)$]{Petrich-1984}).

\smallskip

Next, we shall describe the construction which is introduced in \cite{Gutik-Mykhalenych=2020}.

Let $\boldsymbol{B}_{\omega}$ be the bicyclic monoid and $\mathscr{F}$ be an ${\omega}$-closed subfamily of $\mathscr{P}(\omega)$. On the set $\boldsymbol{B}_{\omega}\times\mathscr{F}$ we define the semigroup operation ``$\cdot$'' in the following way
\begin{equation}\label{eq-1.2}
  (i_1,j_1,F_1)\cdot(i_2,j_2,F_2)=
  \left\{
    \begin{array}{ll}
      (i_1-j_1+i_2,j_2,(j_1-i_2+F_1)\cap F_2), & \hbox{if~} j_1\leqslant i_2;\\
      (i_1,j_1-i_2+j_2,F_1\cap (i_2-j_1+F_2)), & \hbox{if~} j_1\geqslant i_2.
    \end{array}
  \right.
\end{equation}
In \cite{Gutik-Mykhalenych=2020} is proved that if the family $\mathscr{F}\subseteq\mathscr{P}(\omega)$ is ${\omega}$-closed then $(\boldsymbol{B}_{\omega}\times\mathscr{F},\cdot)$ is a semigroup. Moreover, if an ${\omega}$-closed family  $\mathscr{F}\subseteq\mathscr{P}(\omega)$ contains the empty set $\varnothing$, then the set
$ %\begin{equation*}
  \boldsymbol{I}=\{(i,j,\varnothing)\colon i,j\in\omega\}
$ %\end{equation*}
is an ideal of the semigroup $(\boldsymbol{B}_{\omega}\times\mathscr{F},\cdot)$. For any ${\omega}$-closed family $\mathscr{F}\subseteq\mathscr{P}(\omega)$ the following semigroup
\begin{equation*}
  \boldsymbol{B}_{\omega}^{\mathscr{F}}=
\left\{
  \begin{array}{ll}
    (\boldsymbol{B}_{\omega}\times\mathscr{F},\cdot)/\boldsymbol{I}, & \hbox{if~} \varnothing\in\mathscr{F};\\
    (\boldsymbol{B}_{\omega}\times\mathscr{F},\cdot), & \hbox{if~} \varnothing\notin\mathscr{F}
  \end{array}
\right.
\end{equation*}
is defined in \cite{Gutik-Mykhalenych=2020}. The semigroup $\boldsymbol{B}_{\omega}^{\mathscr{F}}$ generalizes the bicyclic monoid and the countable semigroup of matrix units. It is proven in \cite{Gutik-Mykhalenych=2020} that $\boldsymbol{B}_{\omega}^{\mathscr{F}}$ is a combinatorial inverse semigroup and Green's relations, the natural partial order on $\boldsymbol{B}_{\omega}^{\mathscr{F}}$ and its set of idempotents are described.
Here, the criteria when the semigroup $\boldsymbol{B}_{\omega}^{\mathscr{F}}$ is simple, $0$-simple, bisimple, $0$-bisimple, or it has the identity, are given.
In particularly in \cite{Gutik-Mykhalenych=2020} it is proved that the semigroup $\boldsymbol{B}_{\omega}^{\mathscr{F}}$ is isomorphic to the semigrpoup of ${\omega}{\times}{\omega}$-matrix units if and only if $\mathscr{F}$ consists of a singleton set and the empty set, and $\boldsymbol{B}_{\omega}^{\mathscr{F}}$ is isomorphic to the bicyclic monoid if and only if $\mathscr{F}$ consists of a non-empty inductive subset of $\omega$.

\smallskip

Group congruences on the semigroup  $\boldsymbol{B}_{\omega}^{\mathscr{F}}$ and its homomorphic retracts  in the case when an ${\omega}$-closed family $\mathscr{F}$ consists of inductive non-empty subsets of $\omega$ are studied in \cite{Gutik-Mykhalenych=2021}. It is proven that a congruence $\mathfrak{C}$ on $\boldsymbol{B}_{\omega}^{\mathscr{F}}$ is a group congruence if and only if its restriction on a subsemigroup of $\boldsymbol{B}_{\omega}^{\mathscr{F}}$, which is isomorphic to the bicyclic semigroup, is not the identity relation. Also in \cite{Gutik-Mykhalenych=2021}, all non-trivial homomorphic retracts and isomorphisms  of the semigroup $\boldsymbol{B}_{\omega}^{\mathscr{F}}$ are described. In \cite{Gutik-Mykhalenych=2022} it is proved that an injective endomorphism $\varepsilon$ of the semigroup $\boldsymbol{B}_{\omega}^{\mathscr{F}}$ is the indentity transformation if and only if  $\varepsilon$ has three distinct fixed points, which is equivalent to existence non-idempotent element $(i,j,[p))\in\boldsymbol{B}_{\omega}^{\mathscr{F}}$ such that  $(i,j,[p))\varepsilon=(i,j,[p))$.

\smallskip

The semigroup  $\boldsymbol{B}_{\mathbb{Z}}^{\mathscr{F}}$ and its group of automorphisms in the case when an ${\omega}$-closed family $\mathscr{F}$ consists of inductive non-empty subsets of the set of integers $\mathbb{Z}$ are studied in \cite{Gutik-Pozdniakova=2021, Gutik-Pozdniakova=2023}.

\smallskip

In \cite{Gutik-Lysetska=2021, Lysetska=2020} the algebraic structure of the semigroup $\boldsymbol{B}_{\omega}^{\mathscr{F}}$ is established in the case when ${\omega}$-closed family $\mathscr{F}$ consists of atomic subsets of ${\omega}$.
Also, in \cite{Gutik-Popadiuk=2022, Gutik-Popadiuk=2023, Popadiuk=2022} the semigroup $\boldsymbol{B}_{\omega}^{\mathscr{F}_n}$ and its semigroup of endomorphisms are described in the case when the ${\omega}$-closed family $\mathscr{F}_n$ is generated by an $n$-elements interval of ${\omega}$.

\smallskip

It is well-known that every automorphism of the bicyclic monoid $\boldsymbol{B}_{\omega}$  is the identity self-map of $\boldsymbol{B}_{\omega}$ \cite{Clifford-Preston-1961}, and hence the group $\mathbf{Aut}(\boldsymbol{B}_{\omega})$ of automorphisms of $\boldsymbol{B}_{\omega}$ is trivial. In \cite{Gutik-Prokhorenkova-Sekh=2021} it is proved that the semigroups $\mathrm{\mathbf{End}}(\boldsymbol{B}_{\omega})$ of the endomorphisms of the bicyclic semigroup $\boldsymbol{B}_{\omega}$ is isomorphic to the semidirect products $(\omega,+)\rtimes_\varphi(\omega,*)$, where $+$ and $*$ are the usual addition and the usual multiplication on $\omega$.

\smallskip

Later we assume that an ${\omega}$-closed family $\mathscr{F}$ consists of nonempty inductive nonempty subsets of $\omega$.

In this paper we study injective endomorphisms of the semigroup $\boldsymbol{B}_{\omega}^{\mathscr{F}}$ for the family $\mathscr{F}=\{[0), [1)\}$. We describe
the elements of the monoid $\boldsymbol{End}^1_*(\boldsymbol{B}_{\omega}^{\mathscr{F}})$ of all injective monoid endomorphisms of the monoid $\boldsymbol{B}_{\omega}^{\mathscr{F}}$, and show that Green's relations $\mathscr{R}$, $\mathscr{L}$, $\mathscr{H}$, $\mathscr{D}$, and $\mathscr{J}$  on  $\boldsymbol{End}^1_*(\boldsymbol{B}_{\omega}^{\mathscr{F}})$ coincide with the relation of equality.

\section{{On injective endomorphisms of the monoid $\boldsymbol{B}_{\omega}^{\mathscr{F}}$}}\label{section-2}

\begin{remark}\label{remark-2.1}
By Proposition~1 of \cite{Gutik-Mykhalenych=2021} for any $\omega$-closed family $\mathscr{F}$ of inductive subsets in $\mathscr{P}(\omega)$ there exists an $\omega$-closed family $\mathscr{F}^*$ of inductive subsets in $\mathscr{P}(\omega)$ such that $[0)\in \mathscr{F}^*$ and the semigroups $\boldsymbol{B}_{\omega}^{\mathscr{F}}$ and $\boldsymbol{B}_{\omega}^{\mathscr{F}^*}$ are isomorphic. Hence without loss of generality we may assume that the family $\mathscr{F}$ contains the set $[0)$.
\end{remark}

If $\mathscr{F}$ is an arbitrary $\omega$-closed family $\mathscr{F}$ of inductive subsets in $\mathscr{P}(\omega)$ and $[s)\in \mathscr{F}$ for some $s\in \omega$ then
\begin{equation*}
  \boldsymbol{B}_{\omega}^{\{[s)\}}=\{(i,j,[s))\colon i,j\in\omega\}
\end{equation*}
is a subsemigroup of $\boldsymbol{B}_{\omega}^{\mathscr{F}}$ \cite{Gutik-Mykhalenych=2021} and by Proposition~3 of \cite{Gutik-Mykhalenych=2020} the semigroup $\boldsymbol{B}_{\omega}^{\{[s)\}}$ is isomorphic to the bicyclic semigroup.

\begin{lemma}\label{lemma-2.2}
Let $\mathscr{F}$ be an arbitrary $\omega$-closed family $\mathscr{F}$ of inductive subsets in $\mathscr{P}(\omega)$ and $\alpha$ be an endomorphism of the monoid $\boldsymbol{B}_{\omega}^{\mathscr{F}}$. Then there exists a positive integer $k$ such that
\begin{equation*}
(i,j,[0))\alpha=(ki,kj,[0))
\end{equation*}
for all $i,j\in\omega$.
\end{lemma}

\begin{proof}
By Proposition~3 of \cite{Gutik-Mykhalenych=2020} the submonoid $\boldsymbol{B}_{\omega}^{\{[0)\}}$ is isomorphic to the bicyclic monoid. Since $(0,0,[0))\alpha=(0,0,[0))$, Proposition~4 of \cite{Gutik-Mykhalenych=2021} implies that
$
(\boldsymbol{B}_{\omega}^{\{[0)\}})\alpha\subseteq \boldsymbol{B}_{\omega}^{\{[0)\}}.
$
Then Lemma~1 of \cite{Gutik-Prokhorenkova-Sekh=2021} there exists a positive integer $k$ such that
$
(i,j,[0))\alpha=(ki,kj,[0))
$
for all $i,j\in\omega$.
\end{proof}

\begin{example}\label{example-2.3}
Let $\mathscr{F}=\{[0),[1)\}$. Fix an arbitrary positive integer $k$ and any $p\in\{0,\ldots,k-1\}$. For all $i,j\in\omega$ we define the transformation $\alpha_{k,p}$  of the semigroup $\boldsymbol{B}_{\omega}^{\mathscr{F}}$ in the following way
\begin{align*}
  (i,j,[0))\alpha_{k,p}&=(ki,kj,[0)), \\
  (i,j,[1))\alpha_{k,p}&=(p+ki,p+kj,[1)).
\end{align*}
It is obvious that $\alpha_{k,p}$ is an injective transformation of the monoid $\boldsymbol{B}_{\omega}^{\mathscr{F}}$.
\end{example}

\begin{lemma}\label{lemma-2.4}
Let $\mathscr{F}=\{[0),[1)\}$. Then for an arbitrary positive integer $k$ and any $p\in\{0,\ldots,k-1\}$ the map $\alpha_{k,p}$ is an injective endomorphism of the monoid $\boldsymbol{B}_{\omega}^{\mathscr{F}}$.
\end{lemma}

\begin{proof}
By Proposition~3 of \cite{Gutik-Mykhalenych=2020} the subsemigroups $\boldsymbol{B}_{\omega}^{\{[0)\}}$ and $\boldsymbol{B}_{\omega}^{\{[1)\}}$ are isomorphic to the bicyclic semigroup. By Lemma~1 of \cite{Gutik-Prokhorenkova-Sekh=2021} the restrictions of $\alpha_{k,p}$ onto the subsemigroups $\boldsymbol{B}_{\omega}^{\{[0)\}}$ and $\boldsymbol{B}_{\omega}^{\{[1)\}}$ are endomorphisms of $\boldsymbol{B}_{\omega}^{\{[0)\}}$ and $\boldsymbol{B}_{\omega}^{\{[1)\}}$, respectively. This implies that for all $i,j,s,t\in \omega$ the following equalities hold
\begin{align*}
  ((i,j,[0))\cdot(s,t,[0)))\alpha_{k,p}&=(i,j,[0))\alpha_{k,p}\cdot (s,t,[0))\alpha_{k,p}, \\
  ((i,j,[1))\cdot(s,t,[1)))\alpha_{k,p}&=(i,j,[1))\alpha_{k,p}\cdot (s,t,[1))\alpha_{k,p}.
\end{align*}

For any $i,j,p,q\in \omega$ we have that
\begin{align*}
  ((i,j,[0))\cdot(s,t,[1)))\alpha_{k,p}&=
  \left\{
    \begin{array}{ll}
      (i+s-j,t,(j-s+[0))\cap[1))\alpha_{k,p}, & \hbox{if~} j<s;\\
      (i,t,[0)\cap[1))\alpha_{k,p},           & \hbox{if~} j=s;\\
      (i,j+t-s,[0)\cap(s-j+[1)))\alpha_{k,p}, & \hbox{if~} j>s
    \end{array}
  \right.
   =\\
   &=
   \left\{
    \begin{array}{ll}
      (i+s-j,t,[1))\alpha_{k,p}, & \hbox{if~} j<s;\\
      (i,t,[1))\alpha_{k,p},     & \hbox{if~} j=s;\\
      (i,j+t-s,[0))\alpha_{k,p}, & \hbox{if~} j>s
    \end{array}
  \right.
   =\\
   &=
   \left\{
    \begin{array}{ll}
      (p+k(i+s-j),p+kt,[1)), & \hbox{if~} j<s;\\
      (p+ki,p+kt,[1)),       & \hbox{if~} j=s;\\
      (ki,k(j+t-s),[0)),     & \hbox{if~} j>s,
    \end{array}
  \right.
\end{align*}
\begin{align*}
  (i,j,[0))&\alpha_{k,p}\cdot (s,t,[1))\alpha_{k,p}=(ki,kj,[0))\cdot(p+ks,p+kt,[1))= \\
   &=
   \left\{
     \begin{array}{ll}
       (ki+p+ks-kj,p+kt,(kj-p-ks+[0))\cap[1)), & \hbox{if~} kj<p+ks;\\
       (ki,p+kt,[0)\cap[1)),                   & \hbox{if~} kj=p+ks;\\
       (ki,kj+p+kt-p-ks,[0)\cap(p+ks-kj+[1))), & \hbox{if~} kj>p+ks
     \end{array}
   \right.
   = \\
   &=
   \left\{
     \begin{array}{ll}
       (p+k(i+s-j),p+kt,[1)),              & \hbox{if~} kj<p+ks;\\
       (ki,p+kt,[1)),                      & \hbox{if~} kj=p+ks;\\
       (ki,k(j+t-s),[0)), & \hbox{if~} kj>p+ks
     \end{array}
   \right.
   = \\
   &=
   \left\{
     \begin{array}{ll}
       (p+k(i+s-j),p+kt,[1)), & \hbox{if~} kj<ks;\\
       (p+ki,p+kt,[1)),       & \hbox{if~} kj=ks;\\
       (ki,kt,[1)),                   & \hbox{if~} kj=p+ks \hbox{~and~} p=0;\\
       \texttt{vagueness},                   & \hbox{if~} kj=p+ks \hbox{~and~} p\neq0;\\
       (ki,k(j+t-s),[0)), & \hbox{if~} kj>ks
     \end{array}
   \right.
   = \\
   &=
   \left\{
    \begin{array}{ll}
      (p+k(i+s-j),p+kt,[1)), & \hbox{if~} j<s;\\
      (p+ki,p+kt,[1)),       & \hbox{if~} j=s;\\
      (ki,k(j+t-s),[0)),     & \hbox{if~} j>s,
    \end{array}
  \right.
\end{align*}
and
\begin{align*}
  ((i,j,[1))\cdot(s,t,[0)))\alpha_{k,p}&=
    \left\{
     \begin{array}{ll}
       (i+s-j,t,(j-s+[1))\cap[0))\alpha_{k,p}, & \hbox{if~} j<s;\\
       (i,t,[1)\cap[0))\alpha_{k,p},           & \hbox{if~} j=s;\\
       (i,j+t-s,[1)\cap(s-j+[0)))\alpha_{k,p}, & \hbox{if~} j>s
     \end{array}
   \right.
   = \\
   &=
   \left\{
     \begin{array}{ll}
       (i+s-j,t,[0))\alpha_{k,p}, & \hbox{if~} j<s;\\
       (i,t,[1))\alpha_{k,p},     & \hbox{if~} j=s;\\
       (i,j+t-s,[1))\alpha_{k,p}, & \hbox{if~} j>s
     \end{array}
   \right.
   = \\
   &=
   \left\{
     \begin{array}{ll}
       (k(i+s-j),kt,[0)),     & \hbox{if~} j<s;\\
       (p+ki,p+kt,[1)),       & \hbox{if~} j=s;\\
       (p+ki,p+k(j+t-s),[1)), & \hbox{if~} j>s,
     \end{array}
   \right.
\end{align*}
\begin{align*}
  (i,j,[1))&\alpha_{k,p}\cdot (s,t,[0))\alpha_{k,p}=(p+ki,p+kj,[1))\cdot(ks,kt,[0))= \\
   &=
   \left\{
     \begin{array}{ll}
       (p+ki+ks-p-kj,kt,(p+kj-ks+[1))\cap[0)), & \hbox{if~} p+kj<ks; \\
       (p+ki,kt,[1)\cap[0)),                   & \hbox{if~} p+kj=ks; \\
       (p+ki,p+kj+kt-ks,[1)\cap(ks-p-kj+[0))), & \hbox{if~} p+kj>ks
     \end{array}
   \right.
   = \\   
   &=
   \left\{
     \begin{array}{ll}
       (k(i+s-j),kt,[0)),     & \hbox{if~} p+kj<ks; \\
       (p+ki,kt,[1)),         & \hbox{if~} p+kj=ks; \\
       (p+ki,p+k(j+t-s),[1)), & \hbox{if~} p+kj>ks
     \end{array}
   \right.
   = %\\
\end{align*}
\begin{align*}   
   &=
   \left\{
     \begin{array}{ll}
       (k(i+s-j),kt,[0)),     & \hbox{if~} kj<ks; \\
       (ki,kt,[1)),           & \hbox{if~} p+kj=ks \hbox{~and~} p=0; \\
       \texttt{vagueness},    & \hbox{if~} p+kj=ks \hbox{~and~} p\neq0; \\
       (p+ki,p+kt,[1)),       & \hbox{if~} kj=ks; \\
       (p+ki,p+k(j+t-s),[1)), & \hbox{if~} kj>ks
     \end{array}
   \right.
   = \\
   &=
   \left\{
     \begin{array}{ll}
       (k(i+s-j),kt,[0)),     & \hbox{if~} j<s;\\
       (p+ki,p+kt,[1)),       & \hbox{if~} j=s;\\
       (p+ki,p+k(j+t-s),[1)), & \hbox{if~} j>s,
     \end{array}
   \right.
\end{align*}
because $p\in\{0,\ldots,k-1\}$. Thus, $\alpha_{k,p}$ is an endomorphism of the monoid $\boldsymbol{B}_{\omega}^{\mathscr{F}}$.
\end{proof}

\begin{proposition}\label{proposition-2.5}
Let $\mathscr{F}=\{[0),[1)\}$ and $\alpha$ be an injective endomorphism of the monoid $\boldsymbol{B}_{\omega}^{\mathscr{F}}$ such that
\begin{equation*}
(i,j,[0))\alpha=(ki,kj,[0))
\end{equation*}
for all $i,j\in\omega$ and for some positive integer $k$. If $(0,0,[1))\alpha\in\boldsymbol{B}_{\{[1)\}}^{\mathscr{F}}$ then there exists $p\in\{0,\ldots,k-1\}$ such that
\begin{equation*}
(i,j,[1))\alpha=(p+ki,p+kj,[1))
\end{equation*}
for all $i,j\in\omega$, i.e., $\alpha=\alpha_{k,p}$.
\end{proposition}

\begin{proof}
Suppose that $(0,0,[1))\alpha=(p,p,[1))$ for some $p\in \omega$. Since
\begin{equation*}
(1,1,[0))\alpha=(k,k,[0)), \qquad (1,1,[0))\preccurlyeq(0,0,[1))
\end{equation*}
and by Proposition 1.4.21(6) of \cite{Lawson=1998} the endomorphism $\alpha$ preserves the natural partial order, we have that $(k,k,[0))\preccurlyeq(p,p,[1))$. Hence the definition of the natural partial order on $\boldsymbol{B}_{\omega}^{\mathscr{F}}$ (see \cite[Proposition~3]{Gutik-Mykhalenych=2021}) implies that $p\leqslant k-1$.

Suppose that
$
(1,1,[1))\alpha=(p+s,p+s,[1))
$
for some positive integer $s$. Since by Proposition~3 of \cite{Gutik-Mykhalenych=2020} the subsemigroup $\boldsymbol{B}_{\omega}^{\{[1)\}}$ is isomorphic to the bicyclic semigroup, the injectivity of $\alpha$ implies that the image $(\boldsymbol{B}_{\omega}^{\{[1)\}})\alpha$ is isomorphic to the bicyclic semigroup. Put $(0,1,[1))\alpha=(x,y,[1))$. By Proposition 1.4.21 from \cite{Lawson=1998} and Lemma~4 of \cite{Gutik-Mykhalenych=2020} we get that
\begin{align*}
  (1,0,[1))\alpha&=((0,1,[1))^{-1})\alpha
   =((0,1,[1))\alpha)^{-1}
   =(x,y,[1))^{-1}
   =(y,x,[1)).
\end{align*}
This implies that
\begin{align*}
  (p,p,[1))&=(0,0,[1))\alpha=((0,1,[1))\cdot(1,0,[1)))\alpha=\\
   &=(0,1,[1))\alpha\cdot(1,0,[1))\alpha=\\
   &=(x,y,[1))\cdot(y,x,[1))=  \\
   &=(x,x,[1))
\end{align*}
and
\begin{align*}
  (p+s,p+s,[1))&=(1,1,[1))\alpha=\\
   &=((1,0,[1))\cdot(0,1,[1)))\alpha=\\
   &=(1,0,[1))\alpha\cdot(0,1,[1))\alpha= \\
   &=(y,x,[1))\cdot(x,y,[1))=\\
   &=(y,y,[1)),
\end{align*}
and hence by the definition of the semigroup $\boldsymbol{B}_{\omega}^{\mathscr{F}}$ we get that
\begin{equation*}
  (0,1,[1))\alpha=(p,p+s,[1)) \qquad \hbox{and} \qquad (1,0,[1))\alpha=(p+s,p,[1)).
\end{equation*}
Then for any $i,j\in\omega$ we have that
\begin{align*}
  (i,j,[1))\alpha&=((i,0,[1))\cdot(0,j,[1)))\alpha=\\
   &=((1,0,[1))^{i}\cdot(0,1,[1))^{j})\alpha=\\
   &=((1,0,[1))\alpha)^{i}\cdot((0,1,[1))\alpha)^{j}= \\
   &=(p+s,p,[1))^{i}\cdot(p,p+s,[1))^{j}=\\
   &=(p+si,p,[1))\cdot(p,p+sj,[1))=\\
   &=(p+si,p+sj,[1)).
\end{align*}

The definition of the natural partial order on $\boldsymbol{B}_{\omega}^{\mathscr{F}}$ (see \cite[Proposition 2]{Gutik-Mykhalenych=2021}) implies that for any positive integer $t$ we have that
\begin{equation*}
  (t+1,t+1,[1))\preccurlyeq(t+1,t+1,[0))\preccurlyeq (t,t,[1)).
\end{equation*}
Since by Proposition 1.4.21 of \cite{Lawson=1998} the endomorphism $\alpha$ preserves the natural partial order, we conclude that
\begin{equation*}
  (t+1,t+1,[1))\alpha\preccurlyeq(t+1,t+1,[0))\alpha\preccurlyeq (t,t,[1))\alpha.
\end{equation*}
This and the definition of $\alpha$ imply that
\begin{equation*}
  (p+s(t+1),p+s(t+1),[1))\preccurlyeq(k(t+1),k(t+1),[0))\preccurlyeq (p+st,p+st,[1)).
\end{equation*}
Then by the definition of the natural partial order on $\boldsymbol{B}_{\omega}^{\mathscr{F}}$ (see \cite[Proposition 3]{Gutik-Mykhalenych=2021}) we obtain that \begin{equation}\label{eq-2.1}
  p+s(t+1)\geqslant k(t+1) \qquad \hbox{and} \qquad k(t+1)-1\geqslant p+st
\end{equation}
for any positive integer $t$.

We claim that $s=k$. Indeed, if $s<k$ then the first inequality of \eqref{eq-2.1} implies that
\begin{equation*}
  p+(s-k)(t+1)\geqslant 0 \qquad \hbox{for all positive integers} \quad t.
\end{equation*}
But the last inequality is incorrect because $s-k<0$. Similarly, if $s>k$ then the second inequality of \eqref{eq-2.1} implies that
\begin{equation*}
  (k-s)t+k-1-p\geqslant 0 \qquad \hbox{for all positive integers} \quad t.
\end{equation*}
But the last inequality is incorrect because $k-s<0$. Then $\alpha=\alpha_{k,p}$. By Lemma~\ref{lemma-2.4}, the map $\alpha_{k,p}$ is an endomorphism of the monoid $\boldsymbol{B}_{\omega}^{\mathscr{F}}$.
\end{proof}

\begin{example}\label{example-2.6}
Let $\mathscr{F}=\{[0),[1)\}$. Fix an arbitrary positive integer $k\geqslant 2$ and any $p\in\{1,\ldots,k-1\}$. For all $i,j\in\omega$ we define the transformation $\beta_{k,p}$  of the semigroup $\boldsymbol{B}_{\omega}^{\mathscr{F}}$ in the following way
\begin{align*}
  (i,j,[0))\beta_{k,p}&=(ki,kj,[0)), \\
  (i,j,[1))\beta_{k,p}&=(p+ki,p+kj,[0)).
\end{align*}
It is obvious that $\beta_{k,p}$ is an injective transformation of the monoid $\boldsymbol{B}_{\omega}^{\mathscr{F}}$.
\end{example}

\begin{lemma}\label{lemma-2.7}
Let $\mathscr{F}=\{[0),[1)\}$. Then for an arbitrary positive integer $k\geqslant 2$ and any $p\in\{1,\ldots,k-1\}$ the map $\beta_{k,p}$ is an injective endomorphism of the monoid $\boldsymbol{B}_{\omega}^{\mathscr{F}}$.
\end{lemma}

\begin{proof}
By Proposition~3 of \cite{Gutik-Mykhalenych=2020} the subsemigroups $\boldsymbol{B}_{\omega}^{\{[0)\}}$ and $\boldsymbol{B}_{\omega}^{\{[1)\}}$ are isomorphic to the bicyclic semigroup. By Lemma~1 of \cite{Gutik-Prokhorenkova-Sekh=2021}, the restriction of $\beta_{k,p}$ onto the subsemigroup $\boldsymbol{B}_{\omega}^{\{[0)\}}$ is an endomorphisms of $\boldsymbol{B}_{\omega}^{\{[0)\}}$, and the restriction of $\beta_{k,p}$ onto the subsemigroup $\boldsymbol{B}_{\omega}^{\{[1)\}}$ is a homomorphisms of $\boldsymbol{B}_{\omega}^{\{[1)\}}$ into $\boldsymbol{B}_{\omega}^{\{[0)\}}$. This implies that for all $i,j,s,t\in \omega$ the following equalities hold
\begin{align*}
  ((i,j,[0))\cdot(s,t,[0)))\beta_{k,p}&=(i,j,[0))\beta_{k,p}\cdot (s,t,[0))\beta_{k,p}, \\
  ((i,j,[1))\cdot(s,t,[1)))\beta_{k,p}&=(i,j,[1))\beta_{k,p}\cdot (s,t,[1))\beta_{k,p}.
\end{align*}

For any $i,j,p,q\in \omega$ we have that
\begin{align*}
  ((i,j,[0))\cdot(s,t,[1)))\beta_{k,p}&=
  \left\{
    \begin{array}{ll}
      (i+s-j,t,(j-s+[0))\cap[1))\beta_{k,p}, & \hbox{if~} j<s;\\
      (i,t,[0)\cap[1))\beta_{k,p},           & \hbox{if~} j=s;\\
      (i,j+t-s,[0)\cap(s-j+[1)))\beta_{k,p}, & \hbox{if~} j>s
    \end{array}
  \right.
   =\\
   &=
   \left\{
    \begin{array}{ll}
      (i+s-j,t,[1))\beta_{k,p}, & \hbox{if~} j<s;\\
      (i,t,[1))\beta_{k,p},     & \hbox{if~} j=s;\\
      (i,j+t-s,[0))\beta_{k,p}, & \hbox{if~} j>s
    \end{array}
  \right.
   =\\
   &=
   \left\{
    \begin{array}{ll}
      (p+k(i+s-j),p+kt,[0)), & \hbox{if~} j<s;\\
      (p+ki,p+kt,[0)),       & \hbox{if~} j=s;\\
      (ki,k(j+t-s),[0)),     & \hbox{if~} j>s,
    \end{array}
  \right.
\end{align*}
\begin{align*}
  (i,j,[0))&\beta_{k,p}\cdot (s,t,[1))\beta_{k,p}=(ki,kj,[0))\cdot(p+ks,p+kt,[0))= \\
   &=
   \left\{
     \begin{array}{ll}
       (ki+p+ks-kj,p+kt,(kj-p-ks+[0))\cap[0)), & \hbox{if~} kj<p+ks;\\
       (ki,p+kt,[0)\cap[0)),                   & \hbox{if~} kj=p+ks;\\
       (ki,kj+p+kt-p-ks,[0)\cap(p+ks-kj+[0))), & \hbox{if~} kj>p+ks
     \end{array}
   \right.
   = \\
%\end{align*}
%\begin{align*}
   &=
   \left\{
     \begin{array}{ll}
       (p+k(i+s-j),p+kt,[0)),              & \hbox{if~} kj<p+ks;\\
       (ki,p+kt,[0)),                      & \hbox{if~} kj=p+ks;\\
       (ki,k(j+t-s),[0)),                  & \hbox{if~} kj>p+ks
     \end{array}
   \right.
   = \\
   &=
   \left\{
     \begin{array}{ll}
       (p+k(i+s-j),p+kt,[0)), & \hbox{if~} kj<ks;\\
       (p+ki,p+kt,[0)),       & \hbox{if~} kj=ks;\\
       (ki,kt,[0)),           & \hbox{if~} kj=p+ks \hbox{~and~} p=0;\\
       \texttt{vagueness},    & \hbox{if~} kj=p+ks \hbox{~and~} p\neq0;\\
       (ki,k(j+t-s),[0))),    & \hbox{if~} kj>ks
     \end{array}
   \right.
   = \\
   &=
   \left\{
    \begin{array}{ll}
      (p+k(i+s-j),p+kt,[0)), & \hbox{if~} j<s;\\
      (p+ki,p+kt,[0)),       & \hbox{if~} j=s;\\
      (ki,k(j+t-s),[0)),     & \hbox{if~} j>s,
    \end{array}
  \right.
\end{align*}
and
\begin{align*}
  ((i,j,[1))\cdot(s,t,[0)))\beta_{k,p}&=
    \left\{
     \begin{array}{ll}
       (i+s-j,t,(j-s+[1)\cap[0))\beta_{k,p},  & \hbox{if~} j<s;\\
       (i,t,[1)\cap[0))\beta_{k,p},           & \hbox{if~} j=s;\\
       (i,j+t-s,[1)\cap(s-j+[0)))\beta_{k,p}, & \hbox{if~} j>s
     \end{array}
   \right.
   = \\
   &=
   \left\{
     \begin{array}{ll}
       (i+s-j,t,[0))\beta_{k,p}, & \hbox{if~} j<s;\\
       (i,t,[1))\beta_{k,p},     & \hbox{if~} j=s;\\
       (i,j+t-s,[1))\beta_{k,p}, & \hbox{if~} j>s
     \end{array}
   \right.
   = \\
   &=
   \left\{
     \begin{array}{ll}
       (k(i+s-j),kt,[0)),     & \hbox{if~} j<s;\\
       (p+ki,p+kt,[0)),       & \hbox{if~} j=s;\\
       (p+ki,p+k(j+t-s),[0)), & \hbox{if~} j>s,
     \end{array}
   \right.
\end{align*}
\begin{align*}
  (i,j,[1))&\beta_{k,p}\cdot (s,t,[0))\beta_{k,p}=(p+ki,p+kj,[0))\cdot(ks,kt,[0))= \\
   &=
   \left\{
     \begin{array}{ll}
       (p+ki+ks-p-kj,kt,(p+kj-ks+[0))\cap[0)), & \hbox{if~} p+kj<ks; \\
       (p+ki,kt,[0)\cap[0)),                   & \hbox{if~} p+kj=ks; \\
       (p+ki,p+kj+kt-ks,[0)\cap(ks-p-kj+[0))), & \hbox{if~} p+kj>ks
     \end{array}
   \right.
   = \\
   &=
   \left\{
     \begin{array}{ll}
       (k(i+s-j),kt,[0)),     & \hbox{if~} p+kj<ks; \\
       (p+ki,kt,[0)),         & \hbox{if~} p+kj=ks; \\
       (p+ki,p+k(j+t-s),[0)), & \hbox{if~} p+kj>ks
     \end{array}
   \right.
   = 
\end{align*}
\begin{align*}   
   &=
   \left\{
     \begin{array}{ll}
       (k(i+s-j),kt,[0)),     & \hbox{if~} kj<ks; \\
       (ki,kt,[0)),           & \hbox{if~} p+kj=ks \hbox{~and~} p=0; \\
       \texttt{vagueness},    & \hbox{if~} p+kj=ks \hbox{~and~} p\neq0; \\
       (p+ki,p+kt,[0)),       & \hbox{if~} kj=ks; \\
       (p+ki,p+k(j+t-s),[0)), & \hbox{if~} kj>ks
     \end{array}
   \right.
   = \\
   &=
   \left\{
     \begin{array}{ll}
       (k(i+s-j),kt,[0)),     & \hbox{if~} j<s;\\
       (p+ki,p+kt,[0)),       & \hbox{if~} j=s;\\
       (p+ki,p+k(j+t-s),[0)), & \hbox{if~} j>s,
     \end{array}
   \right.
\end{align*}
because $p\in\{1,\ldots,k-1\}$. Thus, $\beta_{k,p}$ is an endomorphism of the monoid $\boldsymbol{B}_{\omega}^{\mathscr{F}}$.
\end{proof}

\begin{proposition}\label{proposition-2.8}
Let $\mathscr{F}=\{[0),[1)\}$ and $\beta$ be an injective endomorphism of the monoid $\boldsymbol{B}_{\omega}^{\mathscr{F}}$ such that
\begin{equation*}
(i,j,[0))\beta=(ki,kj,[0))
\end{equation*}
for all $i,j\in\omega$ and for some positive integer $k\geqslant 2$. If $(0,0,[1))\beta\in\boldsymbol{B}^{\{[0)\}}_{\omega}$ then there exists $p\in\{1,\ldots,k-1\}$ such that
\begin{equation*}
(i,j,[1))\beta=(p+ki,p+kj,[0))
\end{equation*}
for all $i,j\in\omega$, i.e., $\beta=\beta_{k,p}$.
\end{proposition}

\begin{proof}
Suppose that $(0,0,[1))\beta=(p,p,[0))$ for some $p\in \omega$. By assumption
$
(0,0,[0))\beta=(0,0,[0))
$
and hence $p\neq 1$. Since
$(1,1,[0))\beta=(k,k,[0))$, 
$(1,1,[0))\preccurlyeq(0,0,[1))$
and by Proposition 1.4.21(6) of \cite{Lawson=1998} the endomorphism $\beta$ preserves the natural partial order, we conclude that $(k,k,[0))\preccurlyeq(p,p,[0))$. Hence the definition of the natural partial order on the monoid $\boldsymbol{B}_{\omega}^{\mathscr{F}}$ (see \cite[Proposition 3]{Gutik-Mykhalenych=2021}) implies that $p\leqslant k-1$.

Suppose that
$
(1,1,[1))\beta=(p+s,p+s,[0))
$
for some positive integer $s$. Since by Proposition~3 of \cite{Gutik-Mykhalenych=2020} the subsemigroup $\boldsymbol{B}_{\omega}^{\{[1)\}}$ is isomorphic to the bicyclic semigroup, the injectivity of $\beta$ implies that the image $(\boldsymbol{B}_{\omega}^{\{[1)\}})\beta$ is isomorphic to the bicyclic semigroup. Put $(0,1,[1))\beta=(x,y,[0))$. By Proposition 1.4.21 from \cite{Lawson=1998} and Lemma~4 of \cite{Gutik-Mykhalenych=2020} we get that
\begin{align*}
 (1,0,[1))\beta&=((0,1,[1))^{-1})\beta
   =((0,1,[1))\beta)^{-1}
   =(x,y,[0))^{-1}
   =(y,x,[0)).
\end{align*}
This implies that
\begin{align*}
  (p,p,[0))&=(0,0,[1))\beta=((0,1,[1))\cdot(1,0,[1)))\beta=\\
   &=(0,1,[1))\beta\cdot(1,0,[1))\beta=\\
   &=(x,y,[0))\cdot(y,x,[0))=  \\
   &=(x,x,[0))
\end{align*}
and
\begin{align*}
  (p+s,p+s,[0))&=(1,1,[1))\beta=\\
   &=((1,0,[1))\cdot(0,1,[1)))\beta=\\
   &=(1,0,[1))\beta\cdot(0,1,[1))\beta= \\
   &=(y,x,[0))\cdot(x,y,[0))=\\
   &=(y,y,[0)),
\end{align*}
and hence by the definition of the semigroup $\boldsymbol{B}_{\omega}^{\mathscr{F}}$ we get that
\begin{equation*}
  (0,1,[1))\beta=(p,p+s,[0)) \qquad \hbox{and} \qquad (1,0,[1))\beta=(p+s,p,[0)).
\end{equation*}
Then for any $i,j\in\omega$ we have that
\begin{align*}
  (i,j,[1))\beta&=((i,0,[1))\cdot(0,j,[1)))\beta=\\
   &=((1,0,[1))^{i}\cdot(0,1,[1))^{j})\beta=\\
   &=((1,0,[1))\beta)^{i}\cdot((0,1,[1))\beta)^{j}= \\
   &=(p+s,p,[0))^{i}\cdot(p,p+s,[0))^{j}=\\
   &=(p+si,p,[0))\cdot(p,p+sj,[0))=\\
   &=(p+si,p+sj,[0)).
\end{align*}

The definition of the natural partial order on $\boldsymbol{B}_{\omega}^{\mathscr{F}}$ (see \cite[Proposition 2]{Gutik-Mykhalenych=2021}) implies that for any positive integer $t$ we have that
\begin{equation*}
  (t+1,t+1,[1))\preccurlyeq(t+1,t+1,[0))\preccurlyeq (t,t,[1)).
\end{equation*}
Since by Proposition 1.4.21 of \cite{Lawson=1998} the endomorphism $\alpha$ preserves the natural partial order, we conclude that
\begin{equation*}
  (t+1,t+1,[1))\beta\preccurlyeq(t+1,t+1,[0))\beta\preccurlyeq (t,t,[1))\beta,
\end{equation*}
which  implies that
\begin{equation*}
  (p+s(t+1),p+s(t+1),[0))\preccurlyeq(k(t+1),k(t+1),[0))\preccurlyeq (p+st,p+st,[0)).
\end{equation*}
Then by the definition of the natural partial order on $\boldsymbol{B}_{\omega}^{\mathscr{F}}$ (see \cite[Proposition 3]{Gutik-Mykhalenych=2021}) we obtain that \begin{equation}\label{eq-2.2}
  p+s(t+1)\geqslant k(t+1)\geqslant p+st
\end{equation}
for any positive integer $t$.

We claim that $s=k$. Indeed, if $s<k$ then the first inequality of \eqref{eq-2.2} implies that
\begin{equation*}
  p+(s-k)(t+1)\geqslant 0 \qquad \hbox{for all positive integers} \quad t.
\end{equation*}
But the last inequality is incorrect because $s-k<0$. Similar, if $s>k$ then the second inequality of \eqref{eq-2.2} implies that
\begin{equation*}
  (k-s)t+k-p\geqslant 0 \qquad \hbox{for all positive integers} \quad t.
\end{equation*}
But the last inequality is incorrect because $k-s<0$.
Hence $\beta=\beta_{k,p}$. By Lemma~\ref{lemma-2.7}, the map $\beta_{k,p}$ is an endomorphism of the monoid $\boldsymbol{B}_{\omega}^{\mathscr{F}}$.
\end{proof}

In the following theorem we describe the elements of the semigroup $\boldsymbol{End}^1_*(\boldsymbol{B}_{\omega}^{\mathscr{F}})$ of all injective monoid endomorphisms of the monoid $\boldsymbol{B}_{\omega}^{\mathscr{F}}$ for the family $\mathscr{F}=\{[0),[1)\}$.

\begin{theorem}\label{theorem-2.9}
Let $\mathscr{F}=\{[0),[1)\}$ and $\varepsilon$ be an injective monoid endomorphism of  $\boldsymbol{B}_{\omega}^{\mathscr{F}}$. Then either there exist a positive integer $k$ and $p\in\{0,\ldots,k-1\}$ such that $\varepsilon=\alpha_{k,p}$ or there exist a positive integer $k\geqslant 2$ and $p\in\{1,\ldots,k-1\}$ such that $\varepsilon=\beta_{k,p}$. Moreover. the following statements hold:
\begin{enumerate}
  \item\label{theorem-2.9(1)} $(i,j,[0))\alpha_{k,p}=(i,j,[0))\beta_{k,p}$ for all $i,j\in\omega$ and any positive integer $k\geqslant 2$ and $p\in\{1,\ldots,k-1\}$;
  \item\label{theorem-2.9(2)} if $k=1$ then $\varepsilon=\alpha_{1,0}$ is an automorphism of the monoid $\boldsymbol{B}_{\omega}^{\mathscr{F}}$ which is the identity selfmap of $\boldsymbol{B}_{\omega}^{\mathscr{F}}$;
  \item\label{theorem-2.9(3)} $\alpha_{k_1,p_1}\alpha_{k_2,p_2}=\alpha_{k_1k_2,p_2+k_2p_1}$ for all positive integers $k_1, k_2$, any $p_1{\in}\{0,\ldots,k_1{-}1\}$, and  any $p_2\in\{0,\ldots,k_2-1\}$;
  \item\label{theorem-2.9(4)} $\alpha_{k_1,p_1}\beta_{k_2,p_2}=\beta_{k_1k_2,p_2+k_2p_1}$ for all positive integers $k_1$ and $k_2\geqslant 2$, any $p_1\in\{0,\ldots,k_1-1\}$, and  any $p_2\in\{1,\ldots,k_2-1\}$;
  \item\label{theorem-2.9(5)} $\beta_{k_1,p_1}\beta_{k_2,p_2}=\beta_{k_1k_2,k_2p_1}$ for all positive integers $k_1,k_2\geqslant 2$, any $p_1{\in}\{1,\ldots,k_1{-}1\}$, and  any $p_2\in\{1,\ldots,k_2-1\}$;
  \item\label{theorem-2.9(6)} $\beta_{k_1,p_1}\alpha_{k_2,p_2}=\beta_{k_1k_2,k_2p_1}$ for all positive integers $k_1\geqslant 2$ and $k_2$, any $p_1\in\{1,\ldots,k_1-1\}$, and  any $p_2\in\{0,\ldots,k_2-1\}$;
  \item\label{theorem-2.9(7)} if $\alpha_{k_2,p_2}$, $\beta_{k_1,p_1}$, and $\beta_{k_2,p_2}$ are well defined elements of $\boldsymbol{End}^1_*(\boldsymbol{B}_{\omega}^{\mathscr{F}})$ then $\beta_{k_1,p_1}\alpha_{k_2,p_2}=\beta_{k_1,p_1}\beta_{k_2,p_2}$;
  \item\label{theorem-2.9(8)} $\alpha_{1,0}$ is the unique idempotent of $\boldsymbol{End}^1_*(\boldsymbol{B}_{\omega}^{\mathscr{F}})$, and moreover $\alpha_{1,0}$ is the identity of $\boldsymbol{End}^1_*(\boldsymbol{B}_{\omega}^{\mathscr{F}})$;
  \item\label{theorem-2.9(9)} $S_{\alpha}=\langle\alpha_{k,p}\mid k\in \mathbb{N}, p\in\{0,\ldots,k-1\}\rangle$ is a cancellative submonoid of $\boldsymbol{End}^1_*(\boldsymbol{B}_{\omega}^{\mathscr{F}})$;
  \item\label{theorem-2.9(10)} $S_{\beta}=\langle\beta_{k,p}\mid k=2,3,\ldots, p\in\{1,\ldots,k-1\}\rangle$ is an ideal in $\boldsymbol{End}^1_*(\boldsymbol{B}_{\omega}^{\mathscr{F}})$.
\end{enumerate}
\end{theorem}

\begin{proof}
Since by Proposition~3 of \cite{Gutik-Mykhalenych=2020} the subsemigroup $\boldsymbol{B}_{\omega}^{\{[0)\}}$ is isomorphic to the bicyclic semigroup, the injectivity of $\varepsilon$ implies that the image $(\boldsymbol{B}_{\omega}^{\{[0)\}})\varepsilon$ is isomorphic to the bicyclic semigroup. The condition $(0,0,[0))\varepsilon=(0,0,[0))$ and Proposition~4 of \cite{Gutik-Mykhalenych=2021} imply that $(\boldsymbol{B}_{\omega}^{\{[0)\}})\varepsilon\subseteq \boldsymbol{B}_{\omega}^{\{[0)\}}$. Similar arguments imply that either $(\boldsymbol{B}_{\omega}^{\{[1)\}})\varepsilon\subseteq \boldsymbol{B}_{\omega}^{\{[1)\}}$ or $(\boldsymbol{B}_{\omega}^{\{[1)\}})\varepsilon\subseteq \boldsymbol{B}_{\omega}^{\{[0)\}}$. Then there  exists a positive integer $k$ such that $(1,1,[0))\varepsilon=(k,k,[0))$. Next we apply either Proposition~\ref{proposition-2.5} or Proposition~\ref{proposition-2.8}. This completes the first statement of the theorem.

\smallskip

\eqref{theorem-2.9(1)} follows from the definitions of endomorphisms $\alpha_{k,p}$ and $\beta_{k,p}$.

\smallskip

\eqref{theorem-2.9(2)} If $k=1$ then $(1,1,[0))\varepsilon=(1,1,[0))$ and $(0,0,[1))\varepsilon=(0,0,[1))$. By Theorem~3 of \cite{Gutik-Mykhalenych=2022}, $\varepsilon$ is the identity selfmap of $\boldsymbol{B}_{\omega}^{\mathscr{F}}$.

\smallskip

\eqref{theorem-2.9(3)} For any $i,j\in \omega$ we have that
\begin{align*}
  (i,j,[1))\alpha_{k_1,p_1}\alpha_{k_2,p_2}&=((i,j,[1))\alpha_{k_1,p_1})\alpha_{k_2,p_2}=\\
   &=(p_1+k_1i,p_1+k_1j,[1))\alpha_{k_2,p_2}=\\
   &=(p_2+k_2(p_1+k_1i),p_2+k_2(p_1+k_1j),[1))= \\
   &=((p_2+k_2p_1)+k_2k_1i,(p_2+k_2p_1)+k_2k_1j,[1)),
\end{align*}
and hence by \eqref{theorem-2.9(1)} we get that $\alpha_{k_1,p_1}\alpha_{k_2,p_2}=\alpha_{k_1k_2,p_2+k_2p_1}$.

\smallskip

\eqref{theorem-2.9(4)} For any $i,j\in \omega$ we have that
\begin{align*}
  (i,j,[1))\alpha_{k_1,p_1}\beta_{k_2,p_2}&=((i,j,[1))\alpha_{k_1,p_1})\beta_{k_2,p_2}=\\
   &=(p_1+k_1i,p_1+k_1j,[1))\beta_{k_2,p_2}=\\
   &=(p_2+k_2(p_1+k_1i),p_2+k_2(p_1+k_1j),[0))= \\
   &=((p_2+k_2p_1)+k_2k_1i,(p_2+k_2p_1)+k_2k_1j,[0)),
\end{align*}
and by \eqref{theorem-2.9(1)} we obtain that $\alpha_{k_1,p_1}\alpha_{k_2,p_2}=\beta_{k_1k_2,p_2+k_2p_1}$.

\smallskip

\eqref{theorem-2.9(5)} For any $i,j\in \omega$ we have that
\begin{align*}
  (i,j,[1))\beta_{k_1,p_1}\beta_{k_2,p_2}&=((i,j,[1))\beta_{k_1,p_1})\beta_{k_2,p_2}=\\
   &=(p_1+k_1i,p_1+k_1j,[0))\beta_{k_2,p_2}=\\
   &=(k_2(p_1+k_1i),k_2(p_1+k_1j),[0))= \\
   &=(k_2p_1+k_2k_1i,k_2p_1+k_2k_1j,[0)).
\end{align*}
Hence by \eqref{theorem-2.9(1)} we get that $\beta_{k_1,p_1}\beta_{k_2,p_2}=\beta_{k_1k_2,k_2p_1}$.

\smallskip

\eqref{theorem-2.9(6)} For any $i,j\in \omega$ we have that
\begin{align*}
  (i,j,[1))\beta_{k_1,p_1}\beta_{k_2,p_2}&=((i,j,[1))\beta_{k_1,p_1})\alpha_{k_2,p_2}=\\
   &=(p_1+k_1i,p_1+k_1j,[0))\alpha_{k_2,p_2}=\\
   &=(k_2(p_1+k_1i),k_2(p_1+k_1j),[0))= \\
   &=(k_2p_1+k_2k_1i,k_2p_1+k_2k_1j,[0)).
\end{align*}
Hence by \eqref{theorem-2.9(1)} we get that $\beta_{k_1,p_1}\alpha_{k_2,p_2}=\beta_{k_1k_2,k_2p_1}$.

\smallskip

\eqref{theorem-2.9(7)} follows from statements \eqref{theorem-2.9(5)} and \eqref{theorem-2.9(6)}.

\smallskip

\eqref{theorem-2.9(8)} The first part follows from the first statement and statements \eqref{theorem-2.9(3)} and \eqref{theorem-2.9(5)}. Next we apply \eqref{theorem-2.9(2)}.

\smallskip

\eqref{theorem-2.9(9)} Statement \eqref{theorem-2.9(3)} implies that $S_\alpha$ is a subsemigroup of $\boldsymbol{End}^1_*(\boldsymbol{B}_{\omega}^{\mathscr{F}})$. Since $\alpha_{1,0}$ is the identity of $\boldsymbol{End}^1_*(\boldsymbol{B}_{\omega}^{\mathscr{F}})$ and $\alpha_{1,0}\in S_\alpha$, $S_\alpha$ is a submonoid of $\boldsymbol{End}^1_*(\boldsymbol{B}_{\omega}^{\mathscr{F}})$.

Suppose that $\alpha_{k,p}\alpha_{k_1,p_1}=\alpha_{k,p}\alpha_{k_2,p_2}$ for some $\alpha_{k,p},\alpha_{k_1,p_1},\alpha_{k_2,p_2}$. Then by \eqref{theorem-2.9(3)} we have that $\alpha_{kk_1,p_1+k_1p}=\alpha_{kk_2,p_2+k_2p}$ and the definition of the monoid $\boldsymbol{B}_{\omega}^{\mathscr{F}}$ implies that $kk_1=kk_2$ and $p_1+k_1p=p_2+k_2p$. The first equality implies that $k_1=k_2$ because $k\geqslant 1$. The equality $p_1+k_1p=p_2+k_2p$ implies that $p_1-p_2=p(k_2-k_1)$, and since $k_1=k_2$ we get that $p_1=p_2$. Hence $S_{\alpha}$ is a left cancellative semigroup.

Suppose that $\alpha_{k_1,p_1}\alpha_{k,p}=\alpha_{k_2,p_2}\alpha_{k,p}$ for some $\alpha_{k,p},\alpha_{k_1,p_1},\alpha_{k_2,p_2}\in S_{\alpha}$. Then by \eqref{theorem-2.9(3)} we have that $\alpha_{k_1k,p+kp_1}=\alpha_{k_2k,p+kp_2}$ and the definition of the monoid $\boldsymbol{B}_{\omega}^{\mathscr{F}}$ implies that $kk_1=kk_2$ and $p+kp_1=p+kp_2$. This implies that $k_1=k_2$ and $p_1=p_2$ because $k\geqslant 1$. Hence $S_{\alpha}$ is a right cancellative semigroup.

\smallskip

\eqref{theorem-2.9(10)} follows from statements \eqref{theorem-2.9(4)}, \eqref{theorem-2.9(5)}, and \eqref{theorem-2.9(6)}.
\end{proof}

\section{{Green's relations on the semigroup $\boldsymbol{End}^1_*(\boldsymbol{B}_{\omega}^{\mathscr{F}})$}}

If $S$ is a semigroup, then we shall denote the Green relations on $S$ by $\mathscr{R}$, $\mathscr{L}$, $\mathscr{J}$, $\mathscr{D}$ and $\mathscr{H}$ (see \cite[Section~2.1]{Clifford-Preston-1961}):
\begin{align*}
    &\qquad a\mathscr{R}b \mbox{ if and only if } aS^1=bS^1;\\
    &\qquad a\mathscr{L}b \mbox{ if and only if } S^1a=S^1b;\\
    &\qquad a\mathscr{J}b \mbox{ if and only if } S^1aS^1=S^1bS^1;\\
    &\qquad \mathscr{D}=\mathscr{L}\circ\mathscr{R}=
          \mathscr{R}\circ\mathscr{L};\\
    &\qquad \mathscr{H}=\mathscr{L}\cap\mathscr{R}.
\end{align*}
%The $\mathscr{R}$-class (resp., $\mathscr{L}$-, $\mathscr{H}$-, $\mathscr{D}$- or $\mathscr{J}$-class) of the semigroup $S$ which contains an element $a$ of $S$ will be denoted by $R_a$ (resp., $L_a$, $H_a$, $D_a$ or $J_a$).

In this section we describe Green's relations  on the semigroup $\boldsymbol{End}^1_*(\boldsymbol{B}_{\omega}^{\mathscr{F}})$ for the family $\mathscr{F}=\{[0),[1)\}$.

Statements \eqref{theorem-2.9(3)}--\eqref{theorem-2.9(6)} of Theorem~\ref{theorem-2.9} imply the following two lemmas.

\begin{lemma}\label{lemma-3.1}
Let $\mathscr{F}=\{[0),[1)\}$. Then  neither $\alpha_{k_1,p_1}\mathscr{R}\beta_{k_2,p_2}$ nor $\alpha_{k_1,p_1}\mathscr{L}\beta_{k_2,p_2}$ holds in the monoid $\boldsymbol{End}^1_*(\boldsymbol{B}_{\omega}^{\mathscr{F}})$ for any $\alpha_{k_1,p_1}\in S_{\alpha}$ and any $\beta_{k_2,p_2}\in S_{\beta}$.
\end{lemma}

\begin{lemma}\label{lemma-3.2}
Let $\mathscr{F}=\{[0),[1)\}$. Then  the relation $\alpha_{k_1,p_1}\mathscr{J}\beta_{k_2,p_2}$ in $\boldsymbol{End}^1_*(\boldsymbol{B}_{\omega}^{\mathscr{F}})$ does not holds for any $\alpha_{k_1,p_1}\in S_{\alpha}$ and any $\beta_{k_2,p_2}\in S_{\beta}$.
\end{lemma}

\begin{proposition}\label{proposition-3.3}
Let $\mathscr{F}=\{[0),[1)\}$. Then the following statements hold:
\begin{enumerate}
  \item\label{theorem-3.3(1)} $\alpha_{k_1,p_1}\mathscr{R}\alpha_{k_2,p_2}$ in $\boldsymbol{End}^1_*(\boldsymbol{B}_{\omega}^{\mathscr{F}})$ if and only if $k_1=k_2$ and $p_1=p_2$;
  \item\label{theorem-3.3(2)} $\alpha_{k_1,p_1}\mathscr{L}\alpha_{k_2,p_2}$ in $\boldsymbol{End}^1_*(\boldsymbol{B}_{\omega}^{\mathscr{F}})$ if and only if $k_1=k_2$ and $p_1=p_2$;
  \item\label{theorem-3.3(3)} $\beta_{k_1,p_1}\mathscr{R}\beta_{k_2,p_2}$ in $\boldsymbol{End}^1_*(\boldsymbol{B}_{\omega}^{\mathscr{F}})$ if and only if $k_1=k_2$ and $p_1=p_2$;
  \item\label{theorem-3.3(4)} $\beta_{k_1,p_1}\mathscr{L}\beta_{k_2,p_2}$ in $\boldsymbol{End}^1_*(\boldsymbol{B}_{\omega}^{\mathscr{F}})$ if and only if  $k_1=k_2$ and $p_1=p_2$;
  \item\label{theorem-3.3(6)} $\alpha_{k_1,p_1}\mathscr{J}\alpha_{k_2,p_2}$ in $\boldsymbol{End}^1_*(\boldsymbol{B}_{\omega}^{\mathscr{F}})$ if and only if $k_1=k_2$ and $p_1=p_2$;
  \item\label{theorem-3.3(7)} $\beta_{k_1,p_1}\mathscr{J}\beta_{k_2,p_2}$ in $\boldsymbol{End}^1_*(\boldsymbol{B}_{\omega}^{\mathscr{F}})$ if and only if $k_1=k_2$ and $p_1=p_2$;
\end{enumerate}
\end{proposition}

\begin{proof}
\eqref{theorem-3.3(1)}  Implication $(\Leftarrow)$ is trivial.

$(\Rightarrow)$ Suppose that $\alpha_{k_1,p_1}\mathscr{R}\alpha_{k_2,p_2}$ in $\boldsymbol{End}^1_*(\boldsymbol{B}_{\omega}^{\mathscr{F}})$. The definition of the Green relation $\mathscr{R}$ implies that there exist $\varepsilon_1,\varepsilon_2\in \boldsymbol{End}^1_*(\boldsymbol{B}_{\omega}^{\mathscr{F}})$ such that
\begin{equation*}
  \alpha_{k_1,p_1}=\alpha_{k_2,p_2}\varepsilon_1 \qquad \hbox{and} \qquad \alpha_{k_2,p_2}=\alpha_{k_1,p_1}\varepsilon_2.
\end{equation*}
By statements \eqref{theorem-2.9(3)} and \eqref{theorem-2.9(4)} of Theorem~\ref{theorem-2.9} we have that $\varepsilon_1,\varepsilon_2\in S_{\alpha}$, i.e., $\varepsilon_1=\alpha_{k_1^{\prime},p_1^{\prime}}$ and $\varepsilon_2=\alpha_{k_2^{\prime},p_2^{\prime}}$ for some positive integers $k_1^{\prime}$ and $k_2^{\prime}$, $p_1^{\prime}\in\{0,\ldots,k_1^{\prime}-1\}$, and $p_2^{\prime}\in\{0,\ldots,k_2^{\prime}-1\}$.

We claim that $\varepsilon_1=\varepsilon_2=\alpha_{1,0}$. Indeed, by Theorem~\ref{theorem-2.9}\eqref{theorem-2.9(3)} we get that
\begin{equation*}
  \alpha_{k_1,p_1}=\alpha_{k_2,p_2}\alpha_{k_1^{\prime},p_1^{\prime}}=\alpha_{k_2k_1^{\prime},p_1^{\prime}+k_1^{\prime}p_2}
\qquad \hbox{and} \qquad
  \alpha_{k_2,p_2}=\alpha_{k_1,p_1}\alpha_{k_2^{\prime},p_2^{\prime}}=\alpha_{k_1k_2^{\prime},p_2^{\prime}+k_2^{\prime}p_1},
\end{equation*}
which implies the following  equalities
\begin{align*}
  k_1&=k_2k_1^{\prime}, \\
  p_1&=p_1^{\prime}+k_1^{\prime}p_2, \\
  k_2&=k_1k_2^{\prime}, \\
  p_2&=p_2^{\prime}+k_2^{\prime}p_1.
\end{align*}
Since $k_1,k_2,k_1^{\prime},k_2^{\prime}$ are positive integers, the equalities $k_1=k_2k_1^{\prime}$ and $k_2=k_1k_2^{\prime}$ imply that $k_1^{\prime}=k_2^{\prime}=1$. Then we have that $p_1=p_1^{\prime}+p_2$ and $p_2=p_2^{\prime}+p_1$ which implies that $p_1^{\prime}=p_2^{\prime}=0$. Then we have that $\varepsilon_1=\varepsilon_2=\alpha_{1,0}$, and hence $k_1=k_2$ and $p_1=p_2$.

\smallskip

The proof of statement \eqref{theorem-3.3(2)} is similar to \eqref{theorem-3.3(1)}.

\smallskip

\eqref{theorem-3.3(3)} Implication $(\Leftarrow)$ is trivial.

$(\Rightarrow)$ Suppose that $\beta_{k_1,p_1}\mathscr{R}\beta_{k_2,p_2}$ in $\boldsymbol{End}^1_*(\boldsymbol{B}_{\omega}^{\mathscr{F}})$. The definition of the Green relation $\mathscr{R}$ implies that there exist $\varepsilon_1,\varepsilon_2\in \boldsymbol{End}^1_*(\boldsymbol{B}_{\omega}^{\mathscr{F}})$ such that
\begin{equation*}
  \beta_{k_1,p_1}=\beta_{k_2,p_2}\varepsilon_1 \qquad \hbox{and} \qquad \beta_{k_2,p_2}=\beta_{k_1,p_1}\varepsilon_2.
\end{equation*}
Then we have that
\begin{enumerate}
  \item[(a)] either $\varepsilon_1=\alpha_{k_1^{\prime},p_1^{\prime}}$ or $\varepsilon_1=\beta_{k_1^{\prime},p_1^{\prime}}$; \; and
  \item[(b)] either $\varepsilon_2=\alpha_{k_2^{\prime},p_2^{\prime}}$ or $\varepsilon_2=\beta_{k_2^{\prime},p_2^{\prime}}$.
\end{enumerate}

By Theorem~\ref{theorem-2.9}\eqref{theorem-2.9(7)} in any above  case we get that
\begin{align*}
  \beta_{k_1,p_1}&=\beta_{k_2,p_2}\varepsilon_1=\beta_{k_2k_1^{\prime},k_1^{\prime}p_2}; \\
  \beta_{k_2,p_2}&=\beta_{k_1,p_1}\varepsilon_2=\beta_{k_1k_2^{\prime},k_2^{\prime}p_1},
\end{align*}
and hence
\begin{align*}
  k_1&=k_2k_1^{\prime}; \\
  p_1&=k_1^{\prime}p_2; \\
  k_2&=k_1k_2^{\prime}; \\
  p_2&=k_2^{\prime}p_1.
\end{align*}
This implies that $p_1=k_1^{\prime}k_2^{\prime}p_1$ and $p_2=k_2^{\prime}k_1^{\prime}p_2$, and hence $k_1^{\prime}k_2^{\prime}=1$. Since $k_1^{\prime}$ and $k_2^{\prime}$ are positive integers, $k_1^{\prime}=k_2^{\prime}=1$, and hence $k_1=k_2$. This implies that $\varepsilon_1=\alpha_{k_1^{\prime},p_1^{\prime}}$ and $\varepsilon_2=\alpha_{k_2^{\prime},p_2^{\prime}}$. Also the  equalities $p_1=k_1^{\prime}p_2$ and $p_2=k_2^{\prime}p_1$ imply that $p_1=p_2$.

\smallskip

\eqref{theorem-3.3(4)} Implication $(\Leftarrow)$ is trivial.

$(\Rightarrow)$ Suppose that $\beta_{k_1,p_1}\mathscr{L}\beta_{k_2,p_2}$ in $\boldsymbol{End}^1_*(\boldsymbol{B}_{\omega}^{\mathscr{F}})$. The definition of the Green relation $\mathscr{L}$ implies that there exist $\varepsilon_1,\varepsilon_2\in \boldsymbol{End}^1_*(\boldsymbol{B}_{\omega}^{\mathscr{F}})$ such that
\begin{equation*}
  \beta_{k_1,p_1}=\varepsilon_1\beta_{k_2,p_2} \qquad \hbox{and} \qquad \beta_{k_2,p_2}=\varepsilon_2\beta_{k_1,p_1}.
\end{equation*}

We consider the possible cases:
\begin{enumerate}
  \item[(a)] $\varepsilon_1=\alpha_{k_1^{\prime},p_1^{\prime}}$ and $\varepsilon_2=\alpha_{k_2^{\prime},p_2^{\prime}}$;
  \item[(b)] $\varepsilon_1=\alpha_{k_1^{\prime},p_1^{\prime}}$ and $\varepsilon_2=\beta_{k_2^{\prime},p_2^{\prime}}$;
  \item[(c)] $\varepsilon_1=\beta_{k_1^{\prime},p_1^{\prime}}$ and $\varepsilon_2=\alpha_{k_2^{\prime},p_2^{\prime}}$;
  \item[(d)] $\varepsilon_1=\beta_{k_1^{\prime},p_1^{\prime}}$ and $\varepsilon_2=\beta_{k_2^{\prime},p_2^{\prime}}$.
\end{enumerate}

In case (a) by Theorem~\ref{theorem-2.9}\eqref{theorem-2.9(4)} we have that
\begin{align*}
  \beta_{k_1,p_1}&=\alpha_{k_1^{\prime},p_1^{\prime}}\beta_{k_2,p_2}= \beta_{k_1^{\prime}k_2,p_2+k_2p_1^{\prime}}; \\
  \beta_{k_2,p_2}&=\alpha_{k_2^{\prime},p_2^{\prime}}\beta_{k_1,p_1}= \beta_{k_2^{\prime}k_1,p_1+k_1p_2^{\prime}},
\end{align*}
which implies the following equalities
\begin{align*}
  k_1&=k_1^{\prime}k_2; \\
  p_1&=p_2+k_2p_1^{\prime}; \\
  k_2&=k_2^{\prime}k_1; \\
  p_2&=p_1+k_1p_2^{\prime}.
\end{align*}
Since $k_1,k_2,k_1^{\prime},k_2^{\prime}$ are positive integers, the equalities $k_1=k_1^{\prime}k_2$ and $k_2=k_2^{\prime}k_1$ imply that $k_1^{\prime}k_2^{\prime}=1$, and hence $k_1^{\prime}=k_2^{\prime}=1$. This implies that $k_1=k_2$. Next, by the equalities $p_1=p_2+k_2p_1^{\prime}$, $p_2=p_1+k_1p_2^{\prime}$, and $k_1=k_2$ we get that $p_1^{\prime}=-p_2^{\prime}$. Since $p_1^{\prime}$ and $p_2^{\prime}$ are nonnegative integers, we obtain that $p_1^{\prime}=p_2^{\prime}=0$, which implies that
\begin{equation*}
\varepsilon_1=\alpha_{k_1^{\prime},p_1^{\prime}}=\varepsilon_2=\alpha_{k_2^{\prime},p_2^{\prime}}=\alpha_{1,0}
\end{equation*}
is the unit element of the monoid $\boldsymbol{End}^1_*(\boldsymbol{B}_{\omega}^{\mathscr{F}})$. Then $\beta_{k_1,p_1}=\beta_{k_2,p_2}$, and the definition of elements of $S_{\beta}$ implies that $k_1=k_2$ and $p_1=p_2$.

Suppose case (b) holds. By statements \eqref{theorem-2.9(4)} and \eqref{theorem-2.9(5)} of Theorem~\ref{theorem-2.9} we get that
\begin{align*}
  \beta_{k_1,p_1}&=\alpha_{k_1^{\prime},p_1^{\prime}}\beta_{k_2,p_2}= \beta_{k_1^{\prime}k_2,p_2+k_2p_1^{\prime}}; \\
  \beta_{k_2,p_2}&=\beta_{k_2^{\prime},p_2^{\prime}}\beta_{k_1,p_1}= \beta_{k_2^{\prime}k_1,k_1p_2^{\prime}},
\end{align*}
which implies the following equalities
\begin{align*}
  k_1&=k_1^{\prime}k_2; \\
  p_1&=p_2+k_2p_1^{\prime}; \\
  k_2&=k_2^{\prime}k_1; \\
  p_2&=k_1p_2^{\prime}.
\end{align*}
Since $k_1,k_2,k_1^{\prime},k_2^{\prime}$ are positive integers, the equalities $k_1=k_1^{\prime}k_2$ and $k_2=k_2^{\prime}k_1$ imply that $k_1^{\prime}k_2^{\prime}=1$, and hence $k_1^{\prime}=k_2^{\prime}=1$. But $k_2^{\prime}\geqslant 2$, a contradiction. The obtained contradiction implies that case (b) is impossible.

Suppose case (c) holds. By statements \eqref{theorem-2.9(4)} and \eqref{theorem-2.9(5)} of Theorem~\ref{theorem-2.9}\eqref{theorem-2.9(4)} we get that
\begin{align*}
  \beta_{k_1,p_1}&=\beta_{k_1^{\prime},p_1^{\prime}}\beta_{k_2,p_2}= \beta_{k_1^{\prime}k_2,k_2p_1^{\prime}}; \\
  \beta_{k_2,p_2}&=\alpha_{k_2^{\prime},p_2^{\prime}}\beta_{k_1,p_1}= \beta_{k_2^{\prime}k_1,p_1+k_1p_2^{\prime}},
\end{align*}
which implies the following equalities
\begin{align*}
  k_1&=k_1^{\prime}k_2; \\
  p_1&=k_2p_1^{\prime}; \\
  k_2&=k_2^{\prime}k_1; \\
  p_2&=p_1+k_1p_2^{\prime}.
\end{align*}
Since $k_1,k_2,k_1^{\prime},k_2^{\prime}$ are positive integers, the equalities $k_1=k_1^{\prime}k_2$ and $k_2=k_2^{\prime}k_1$ imply that $k_1^{\prime}k_2^{\prime}=1$, and hence $k_1^{\prime}=k_2^{\prime}=1$. But $k_1^{\prime}\geqslant 2$, a contradiction. The obtained contradiction implies that case (c) does not hold.

Suppose case (d) holds. By Theorem~\ref{theorem-2.9}\eqref{theorem-2.9(5)} we have that
\begin{align*}
  \beta_{k_1,p_1}&=\beta_{k_1^{\prime},p_1^{\prime}}\beta_{k_2,p_2}= \beta_{k_1^{\prime}k_2,k_2p_1^{\prime}}; \\
  \beta_{k_2,p_2}&=\beta_{k_2^{\prime},p_2^{\prime}}\beta_{k_1,p_1}= \beta_{k_2^{\prime}k_1,k_1p_2^{\prime}},
\end{align*}
which implies the following equalities
\begin{align*}
  k_1&=k_1^{\prime}k_2; \\
  p_1&=k_2p_1^{\prime}; \\
  k_2&=k_2^{\prime}k_1; \\
  p_2&=k_1p_2^{\prime}.
\end{align*}
Since $k_1,k_2,k_1^{\prime},k_2^{\prime}$ are positive integers, the equalities $k_1=k_1^{\prime}k_2$ and $k_2=k_2^{\prime}k_1$ imply that $k_1^{\prime}k_2^{\prime}=1$, and hence $k_1^{\prime}=k_2^{\prime}=1$. But $k_1^{\prime}\geqslant 2$ and $k_2^{\prime}\geqslant 2$, a contradiction. The obtained contradiction implies that case (d) does not hold.

\smallskip

\eqref{theorem-3.3(6)}   Implication $(\Leftarrow)$ is trivial.

$(\Rightarrow)$ Suppose that $\alpha_{k_1,p_1}\mathscr{J}\alpha_{k_2,p_2}$ in $\boldsymbol{End}^1_*(\boldsymbol{B}_{\omega}^{\mathscr{F}})$. The definition of the Green relation $\mathscr{J}$ implies that there exist $\varepsilon_1^{\prime},\varepsilon_1^{\prime\prime},\varepsilon_2^{\prime},\varepsilon_2^{\prime\prime}\in \boldsymbol{End}^1_*(\boldsymbol{B}_{\omega}^{\mathscr{F}})$ such that
\begin{equation*}
  \alpha_{k_1,p_1}=\varepsilon_1^{\prime}\alpha_{k_2,p_2}\varepsilon_1^{\prime\prime} \qquad \hbox{and} \qquad \alpha_{k_2,p_2}=\varepsilon_2^{\prime}\alpha_{k_1,p_1}\varepsilon_2^{\prime\prime}.
\end{equation*}
By statements \eqref{theorem-2.9(3)} and \eqref{theorem-2.9(4)} of Theorem~\ref{theorem-2.9} we have that $\varepsilon_1^{\prime},\varepsilon_1^{\prime\prime},\varepsilon_2^{\prime},\varepsilon_2^{\prime\prime}\in S_{\alpha}$, i.e., $\varepsilon_1^{\prime}=\alpha_{k_1^{\prime},p_1^{\prime}}$, $\varepsilon_1^{\prime\prime}=\alpha_{k_1^{\prime\prime},p_1^{\prime\prime}}$, $\varepsilon_2^{\prime}=\alpha_{k_2^{\prime},p_2^{\prime}}$, and $\varepsilon_2^{\prime\prime}=\alpha_{k_2^{\prime\prime},p_2^{\prime\prime}}$ for some positive integers $k_1^{\prime}$, $k_1^{\prime\prime}$, $k_2^{\prime}$ and $k_2^{\prime\prime}$, $p_1^{\prime}\in\{0,\ldots,k_1^{\prime}-1\}$, $p_1^{\prime\prime}\in\{0,\ldots,k_1^{\prime\prime}-1\}$, $p_2^{\prime}\in\{0,\ldots,k_2^{\prime}-1\}$, and $p_2^{\prime\prime}\in\{0,\ldots,k_2^{\prime\prime}-1\}$.

By Theorem~\ref{theorem-2.9}\eqref{theorem-2.9(3)} we have that
\begin{align*}
  \alpha_{k_1,p_1}&=\alpha_{k_1^{\prime},p_1^{\prime}}\alpha_{k_2,p_2}\alpha_{k_1^{\prime\prime},p_1^{\prime\prime}}
   =\alpha_{k_1^{\prime}k_2,p_2+k_2p_1^{\prime}}\alpha_{k_1^{\prime\prime},p_1^{\prime\prime}}
   =\alpha_{k_1^{\prime}k_2k_1^{\prime\prime},p_1^{\prime\prime}+k_1^{\prime\prime}p_2+k_1^{\prime\prime}k_2p_1^{\prime}};\\
  \alpha_{k_2,p_2}&=\alpha_{k_2^{\prime},p_2^{\prime}}\alpha_{k_1,p_1}\alpha_{k_2^{\prime\prime},p_2^{\prime\prime}}
   =\alpha_{k_2^{\prime}k_1,p_1+k_1p_2^{\prime}}\alpha_{k_2^{\prime\prime},p_2^{\prime\prime}}
   =\alpha_{k_2^{\prime}k_1k_2^{\prime\prime},p_2^{\prime\prime}+k_2^{\prime\prime}p_1+k_2^{\prime\prime}k_1p_2^{\prime}},
\end{align*}
which implies the following equalities
\begin{align*}
  k_1&=k_1^{\prime}k_2k_1^{\prime\prime}; \\
  p_1&=p_1^{\prime\prime}+k_1^{\prime\prime}p_2+k_1^{\prime\prime}k_2p_1^{\prime}; \\
  k_2&=k_2^{\prime}k_1k_2^{\prime\prime}; \\
  p_2&=p_2^{\prime\prime}+k_2^{\prime\prime}p_1+k_2^{\prime\prime}k_1p_2^{\prime}.
\end{align*}
Since $k_1,k_2,k_1^{\prime},k_2^{\prime},k_1^{\prime\prime},k_2^{\prime\prime}$ are positive integers, the equalities $k_1=k_1^{\prime}k_2k_1^{\prime\prime}$ and $k_2=k_2^{\prime}k_1k_2^{\prime\prime}$ imply that $k_1^{\prime}k_2^{\prime}k_1^{\prime\prime}k_2^{\prime\prime}=1$, and hence we have
\begin{equation*}
k_1^{\prime}=k_2^{\prime}=k_1^{\prime\prime}=k_2^{\prime\prime}=1.
\end{equation*}
This implies that $k_1=k_2$. Hence the equalities $k_1^{\prime}=k_2^{\prime}=k_1^{\prime\prime}=k_2^{\prime\prime}=1$, $k_1=k_2$ and the equalities
\begin{align*}
  p_1&=p_1^{\prime\prime}+k_1^{\prime\prime}p_2+k_1^{\prime\prime}k_2p_1^{\prime}; \\
  p_2&=p_2^{\prime\prime}+k_2^{\prime\prime}p_1+k_2^{\prime\prime}k_1p_2^{\prime}
\end{align*}
imply that
\begin{align*}
  p_1&=p_1^{\prime\prime}+p_2+k_1p_1^{\prime}; \\
  p_2&=p_2^{\prime\prime}+p_1+k_1p_2^{\prime}.
\end{align*}
Hence we get that
\begin{equation*}
  p_1^{\prime\prime}+p_2^{\prime\prime}+k_1p_2^{\prime}+k_1p_1^{\prime}=0.
\end{equation*}
Since $k_1$ is a positive integer, by the above equality we have that
\begin{equation*}
  p_1^{\prime\prime}=p_2^{\prime\prime}=p_2^{\prime}=p_1^{\prime}=0.
\end{equation*}
This implies that
\begin{equation*}
  \varepsilon_1^{\prime}=\varepsilon_1^{\prime\prime}=\varepsilon_2^{\prime}=\varepsilon_2^{\prime\prime}=\alpha_{1,0},
\end{equation*}
and hence we have that $\alpha_{k_1,p_1}=\alpha_{k_2,p_2}$, which implies the requested equalities $k_1=k_2$ and $p_1=p_2$.

\smallskip

\eqref{theorem-3.3(7)}   Implication $(\Leftarrow)$ is trivial.

$(\Rightarrow)$ Suppose that $\beta_{k_1,p_1}\mathscr{J}\beta_{k_2,p_2}$ in $\boldsymbol{End}^1_*(\boldsymbol{B}_{\omega}^{\mathscr{F}})$. The definition of the Green relation $\mathscr{J}$ implies that there exist $\varepsilon_1^{\prime},\varepsilon_1^{\prime\prime},\varepsilon_2^{\prime},\varepsilon_2^{\prime\prime}\in \boldsymbol{End}^1_*(\boldsymbol{B}_{\omega}^{\mathscr{F}})$ such that
\begin{equation*}
  \beta_{k_1,p_1}=\varepsilon_1^{\prime}\beta_{k_2,p_2}\varepsilon_1^{\prime\prime} \qquad \hbox{and} \qquad \beta_{k_2,p_2}=\varepsilon_2^{\prime}\beta_{k_1,p_1}\varepsilon_2^{\prime\prime}.
\end{equation*}

We consider the possible cases:
\begin{itemize}
  \item[(a)] $\varepsilon_1^{\prime}=\alpha_{k_1^{\prime},p_1^{\prime}}$, $\varepsilon_1^{\prime\prime}=\alpha_{k_1^{\prime\prime},p_1^{\prime\prime}}$, $\varepsilon_2^{\prime}=\alpha_{k_2^{\prime},p_2^{\prime}}$, and $\varepsilon_2^{\prime\prime}=\alpha_{k_2^{\prime\prime},p_2^{\prime\prime}}$;

  \item[(b)] $\varepsilon_1^{\prime}=\alpha_{k_1^{\prime},p_1^{\prime}}$, $\varepsilon_1^{\prime\prime}=\alpha_{k_1^{\prime\prime},p_1^{\prime\prime}}$, $\varepsilon_2^{\prime}=\alpha_{k_2^{\prime},p_2^{\prime}}$, and $\varepsilon_2^{\prime\prime}=\beta_{k_2^{\prime\prime},p_2^{\prime\prime}}$;

  \item[(c)] $\varepsilon_1^{\prime}=\alpha_{k_1^{\prime},p_1^{\prime}}$, $\varepsilon_1^{\prime\prime}=\alpha_{k_1^{\prime\prime},p_1^{\prime\prime}}$, $\varepsilon_2^{\prime}=\beta_{k_2^{\prime},p_2^{\prime}}$, and $\varepsilon_2^{\prime\prime}=\alpha_{k_2^{\prime\prime},p_2^{\prime\prime}}$;

  \item[(d)] $\varepsilon_1^{\prime}=\alpha_{k_1^{\prime},p_1^{\prime}}$, $\varepsilon_1^{\prime\prime}=\alpha_{k_1^{\prime\prime},p_1^{\prime\prime}}$, $\varepsilon_2^{\prime}=\beta_{k_2^{\prime},p_2^{\prime}}$, and $\varepsilon_2^{\prime\prime}=\beta_{k_2^{\prime\prime},p_2^{\prime\prime}}$;

  \item[(e)] $\varepsilon_1^{\prime}=\alpha_{k_1^{\prime},p_1^{\prime}}$, $\varepsilon_1^{\prime\prime}=\beta_{k_1^{\prime\prime},p_1^{\prime\prime}}$, $\varepsilon_2^{\prime}=\alpha_{k_2^{\prime},p_2^{\prime}}$, and $\varepsilon_2^{\prime\prime}=\alpha_{k_2^{\prime\prime},p_2^{\prime\prime}}$;

  \item[(f)] $\varepsilon_1^{\prime}=\alpha_{k_1^{\prime},p_1^{\prime}}$, $\varepsilon_1^{\prime\prime}=\beta_{k_1^{\prime\prime},p_1^{\prime\prime}}$, $\varepsilon_2^{\prime}=\alpha_{k_2^{\prime},p_2^{\prime}}$, and $\varepsilon_2^{\prime\prime}=\beta_{k_2^{\prime\prime},p_2^{\prime\prime}}$;

  \item[(g)] $\varepsilon_1^{\prime}=\alpha_{k_1^{\prime},p_1^{\prime}}$, $\varepsilon_1^{\prime\prime}=\beta_{k_1^{\prime\prime},p_1^{\prime\prime}}$, $\varepsilon_2^{\prime}=\beta_{k_2^{\prime},p_2^{\prime}}$, and $\varepsilon_2^{\prime\prime}=\alpha_{k_2^{\prime\prime},p_2^{\prime\prime}}$;

  \item[(h)] $\varepsilon_1^{\prime}=\alpha_{k_1^{\prime},p_1^{\prime}}$, $\varepsilon_1^{\prime\prime}=\beta_{k_1^{\prime\prime},p_1^{\prime\prime}}$, $\varepsilon_2^{\prime}=\beta_{k_2^{\prime},p_2^{\prime}}$, and $\varepsilon_2^{\prime\prime}=\beta_{k_2^{\prime\prime},p_2^{\prime\prime}}$;

  \item[(i)] $\varepsilon_1^{\prime}=\beta_{k_1^{\prime},p_1^{\prime}}$, $\varepsilon_1^{\prime\prime}=\alpha_{k_1^{\prime\prime},p_1^{\prime\prime}}$, $\varepsilon_2^{\prime}=\alpha_{k_2^{\prime},p_2^{\prime}}$, and $\varepsilon_2^{\prime\prime}=\alpha_{k_2^{\prime\prime},p_2^{\prime\prime}}$;

  \item[(j)] $\varepsilon_1^{\prime}=\beta_{k_1^{\prime},p_1^{\prime}}$, $\varepsilon_1^{\prime\prime}=\alpha_{k_1^{\prime\prime},p_1^{\prime\prime}}$, $\varepsilon_2^{\prime}=\alpha_{k_2^{\prime},p_2^{\prime}}$, and $\varepsilon_2^{\prime\prime}=\beta_{k_2^{\prime\prime},p_2^{\prime\prime}}$;

  \item[(k)] $\varepsilon_1^{\prime}=\beta_{k_1^{\prime},p_1^{\prime}}$, $\varepsilon_1^{\prime\prime}=\alpha_{k_1^{\prime\prime},p_1^{\prime\prime}}$, $\varepsilon_2^{\prime}=\beta_{k_2^{\prime},p_2^{\prime}}$, and $\varepsilon_2^{\prime\prime}=\alpha_{k_2^{\prime\prime},p_2^{\prime\prime}}$;

  \item[(l)] $\varepsilon_1^{\prime}=\beta_{k_1^{\prime},p_1^{\prime}}$, $\varepsilon_1^{\prime\prime}=\alpha_{k_1^{\prime\prime},p_1^{\prime\prime}}$, $\varepsilon_2^{\prime}=\beta_{k_2^{\prime},p_2^{\prime}}$, and $\varepsilon_2^{\prime\prime}=\beta_{k_2^{\prime\prime},p_2^{\prime\prime}}$;

  \item[(m)] $\varepsilon_1^{\prime}=\beta_{k_1^{\prime},p_1^{\prime}}$, $\varepsilon_1^{\prime\prime}=\beta_{k_1^{\prime\prime},p_1^{\prime\prime}}$, $\varepsilon_2^{\prime}=\alpha_{k_2^{\prime},p_2^{\prime}}$, and $\varepsilon_2^{\prime\prime}=\alpha_{k_2^{\prime\prime},p_2^{\prime\prime}}$;

  \item[(n)] $\varepsilon_1^{\prime}=\beta_{k_1^{\prime},p_1^{\prime}}$, $\varepsilon_1^{\prime\prime}=\beta_{k_1^{\prime\prime},p_1^{\prime\prime}}$, $\varepsilon_2^{\prime}=\alpha_{k_2^{\prime},p_2^{\prime}}$, and $\varepsilon_2^{\prime\prime}=\beta_{k_2^{\prime\prime},p_2^{\prime\prime}}$;

  \item[(o)] $\varepsilon_1^{\prime}=\beta_{k_1^{\prime},p_1^{\prime}}$, $\varepsilon_1^{\prime\prime}=\beta_{k_1^{\prime\prime},p_1^{\prime\prime}}$, $\varepsilon_2^{\prime}=\beta_{k_2^{\prime},p_2^{\prime}}$, and $\varepsilon_2^{\prime\prime}=\alpha_{k_2^{\prime\prime},p_2^{\prime\prime}}$.
\end{itemize}

In case (a) by statements \eqref{theorem-2.9(6)} and \eqref{theorem-2.9(4)} of Theorem~\ref{theorem-2.9} we have that
\begin{align*}
  \beta_{k_1,p_1}&=\alpha_{k_1^{\prime},p_1^{\prime}}\beta_{k_2,p_2}\alpha_{k_1^{\prime\prime},p_1^{\prime\prime}} 
   =\alpha_{k_1^{\prime},p_1^{\prime}}\beta_{k_2k_1^{\prime\prime},p_2k_1^{\prime\prime}} 
   =\beta_{k_1^{\prime}k_2k_1^{\prime\prime},p_2k_1^{\prime\prime}+k_2k_1^{\prime\prime}p_1^{\prime}}; \\
  \beta_{k_2,p_2}&=\alpha_{k_2^{\prime},p_1^{\prime}}\beta_{k_1,p_1}\alpha_{k_2^{\prime\prime},p_1^{\prime\prime}}
   =\alpha_{k_2^{\prime},p_2^{\prime}}\beta_{k_1k_2^{\prime\prime},p_1k_2^{\prime\prime}}
   =\beta_{k_2^{\prime}k_1k_2^{\prime\prime},p_1k_2^{\prime\prime}+k_1k_2^{\prime\prime}p_2^{\prime}},
\end{align*}
which implies the following equalities
\begin{align*}
  k_1&=k_1^{\prime}k_2k_1^{\prime\prime}; \\
  p_1&=p_2k_1^{\prime\prime}+k_2k_1^{\prime\prime}p_1^{\prime}; \\
  k_2&=k_2^{\prime}k_1k_2^{\prime\prime}; \\
  p_2&=p_1k_2^{\prime\prime}+k_1k_2^{\prime\prime}p_2^{\prime}.
\end{align*}
Since $k_1,k_2,k_1^{\prime},k_2^{\prime},k_1^{\prime\prime},k_2^{\prime\prime}$ are positive integers, the equalities $k_1=k_1^{\prime}k_2k_1^{\prime\prime}$ and $k_2=k_2^{\prime}k_1k_2^{\prime\prime}$ imply that $k_1^{\prime}k_2^{\prime}k_1^{\prime\prime}k_2^{\prime\prime}=1$, and hence we have
\begin{equation*}
k_1^{\prime}=k_2^{\prime}=k_1^{\prime\prime}=k_2^{\prime\prime}=1.
\end{equation*}
This implies that $k_1=k_2$. Hence the equalities $k_1^{\prime}=k_2^{\prime}=k_1^{\prime\prime}=k_2^{\prime\prime}=1$, $k_1=k_2$ and the equalities
\begin{align*}
  p_1&=p_2k_1^{\prime\prime}+k_2k_1^{\prime\prime}p_1^{\prime}; \\
  p_2&=p_1k_2^{\prime\prime}+k_1k_2^{\prime\prime}p_2^{\prime}
\end{align*}
imply that
\begin{align*}
  p_1&=p_2+k_1p_1^{\prime}; \\
  p_2&=p_1+k_1p_2^{\prime}.
\end{align*}
Then we have that
\begin{equation*}
  p_1+p_2=p_1+p_2+k_1p_1^{\prime}+k_1p_2^{\prime},
\end{equation*}
and hence $k_1(p_1^{\prime}+p_2^{\prime})=0$. Since $k_1$ is a positive integer, we get that $p_1^{\prime}=-p_2^{\prime}$. Next, the inequalities $p_1^{\prime}\geqslant 0$ and $p_2^{\prime}\geqslant 0$ imply that $p_1^{\prime}=p_2^{\prime}=0$, and by the equality $p_1=p_2+k_1p_1^{\prime}$ we have that $p_1=p_2$. This completes the proof in case (a).

In case (b) by statements \eqref{theorem-2.9(6)} and \eqref{theorem-2.9(5)} of Theorem~\ref{theorem-2.9} we have that
\begin{align*}
  \beta_{k_1,p_1}&=\alpha_{k_1^{\prime},p_1^{\prime}}\beta_{k_2,p_2}\alpha_{k_1^{\prime\prime},p_1^{\prime\prime}} 
   =\alpha_{k_1^{\prime},p_1^{\prime}}\beta_{k_2k_1^{\prime\prime},p_2k_1^{\prime\prime}} 
   =\beta_{k_1^{\prime}k_2k_1^{\prime\prime},p_2k_1^{\prime\prime}+k_2k_1^{\prime\prime}p_1^{\prime}}; \\
  \beta_{k_2,p_2}&=\alpha_{k_2^{\prime},p_1^{\prime}}\beta_{k_1,p_1}\beta_{k_2^{\prime\prime},p_1^{\prime\prime}} 
   =\alpha_{k_2^{\prime},p_2^{\prime}}\beta_{k_1k_2^{\prime\prime},p_1k_2^{\prime\prime}} 
   =\beta_{k_2^{\prime}k_1k_2^{\prime\prime},p_1k_2^{\prime\prime}+k_1k_2^{\prime\prime}p_2^{\prime}},
\end{align*}
which implies the following equalities
\begin{align*}
  k_1&=k_1^{\prime}k_2k_1^{\prime\prime}; \\
  p_1&=p_2k_1^{\prime\prime}+k_2k_1^{\prime\prime}p_1^{\prime}; \\
  k_2&=k_2^{\prime}k_1k_2^{\prime\prime}; \\
  p_2&=p_1k_2^{\prime\prime}+k_1k_2^{\prime\prime}p_2^{\prime}.
\end{align*}
By similar arguments as in case (a) we obtain that $k_1=k_2$ and $p_1=p_2$.

In case (c) by statements \eqref{theorem-2.9(6)}, \eqref{theorem-2.9(4)} and \eqref{theorem-2.9(5)} of Theorem~\ref{theorem-2.9} we have that
\begin{align*}
  \beta_{k_1,p_1}&=\alpha_{k_1^{\prime},p_1^{\prime}}\beta_{k_2,p_2}\alpha_{k_1^{\prime\prime},p_1^{\prime\prime}} 
   =\alpha_{k_1^{\prime},p_1^{\prime}}\beta_{k_2k_1^{\prime\prime},p_2k_1^{\prime\prime}} 
   =\beta_{k_1^{\prime}k_2k_1^{\prime\prime},p_2k_1^{\prime\prime}+k_2k_1^{\prime\prime}p_1^{\prime}}; \\
  \beta_{k_2,p_2}&=\beta_{k_2^{\prime},p_1^{\prime}}\beta_{k_1,p_1}\alpha_{k_2^{\prime\prime},p_1^{\prime\prime}} 
   =\beta_{k_2^{\prime},p_2^{\prime}}\beta_{k_1k_2^{\prime\prime},p_1k_2^{\prime\prime}} 
   =\beta_{k_2^{\prime}k_1k_2^{\prime\prime},k_1k_2^{\prime\prime}p_2^{\prime}},
\end{align*}
which implies the following equalities
\begin{align*}
  k_1&=k_1^{\prime}k_2k_1^{\prime\prime}; \\
  p_1&=p_2k_1^{\prime\prime}+k_2k_1^{\prime\prime}p_1^{\prime}; \\
  k_2&=k_2^{\prime}k_1k_2^{\prime\prime}; \\
  p_2&=k_1k_2^{\prime\prime}p_2^{\prime}.
\end{align*}
Then we have that $k_1=k_1^{\prime}k_2^{\prime}k_1k_2^{\prime\prime}k_1^{\prime\prime}$. The conditions that  $k_1,k_2,k_1^{\prime},k_2^{\prime},k_1^{\prime\prime},k_2^{\prime\prime}$ are positive integers and $k_2^{\prime}\geqslant 2$ contradict the equality $k_1=k_1^{\prime}k_2^{\prime}k_1k_2^{\prime\prime}k_1^{\prime\prime}$, and hence case (c) does not hold.

In case (d) by statements \eqref{theorem-2.9(6)}, \eqref{theorem-2.9(4)} and \eqref{theorem-2.9(5)} of Theorem~\ref{theorem-2.9} we have that
\begin{align*}
  \beta_{k_1,p_1}&=\alpha_{k_1^{\prime},p_1^{\prime}}\beta_{k_2,p_2}\alpha_{k_1^{\prime\prime},p_1^{\prime\prime}}
   =\alpha_{k_1^{\prime},p_1^{\prime}}\beta_{k_2k_1^{\prime\prime},p_2k_1^{\prime\prime}}
   =\beta_{k_1^{\prime}k_2k_1^{\prime\prime},p_2k_1^{\prime\prime}+k_2k_1^{\prime\prime}p_1^{\prime}}; \\
  \beta_{k_2,p_2}&=\beta_{k_2^{\prime},p_1^{\prime}}\beta_{k_1,p_1}\beta_{k_2^{\prime\prime},p_1^{\prime\prime}}
   =\beta_{k_2^{\prime},p_2^{\prime}}\beta_{k_1k_2^{\prime\prime},p_1k_2^{\prime\prime}}
   =\beta_{k_2^{\prime}k_1k_2^{\prime\prime},k_1k_2^{\prime\prime}p_2^{\prime}},
\end{align*}
which implies the following equalities
\begin{align*}
  k_1&=k_1^{\prime}k_2k_1^{\prime\prime}; \\
  p_1&=p_2k_1^{\prime\prime}+k_2k_1^{\prime\prime}p_1^{\prime}; \\
  k_2&=k_2^{\prime}k_1k_2^{\prime\prime}; \\
  p_2&=k_1k_2^{\prime\prime}p_2^{\prime}.
\end{align*}
Next, similar arguments as in case (c) show  that case (d) does not hold.

In case (e) the proof is similar to case (b).

In case (f) by statements \eqref{theorem-2.9(5)} and \eqref{theorem-2.9(4)} of Theorem~\ref{theorem-2.9} we have that
\begin{align*}
  \beta_{k_1,p_1}&=\alpha_{k_1^{\prime},p_1^{\prime}}\beta_{k_2,p_2}\beta_{k_1^{\prime\prime},p_1^{\prime\prime}}
   =\alpha_{k_1^{\prime},p_1^{\prime}}\beta_{k_2k_1^{\prime\prime},p_2k_1^{\prime\prime}}
   =\beta_{k_1^{\prime}k_2k_1^{\prime\prime},p_2k_1^{\prime\prime}+k_2k_1^{\prime\prime}p_1^{\prime}}; \\
  \beta_{k_2,p_2}&=\alpha_{k_2^{\prime},p_1^{\prime}}\beta_{k_1,p_1}\beta_{k_2^{\prime\prime},p_1^{\prime\prime}}
   =\alpha_{k_2^{\prime},p_2^{\prime}}\beta_{k_1k_2^{\prime\prime},p_1k_2^{\prime\prime}}
   =\beta_{k_2^{\prime}k_1k_2^{\prime\prime},p_1k_2^{\prime\prime}+k_1k_2^{\prime\prime}p_2^{\prime}},
\end{align*}
which implies the following equalities
\begin{align*}
  k_1&=k_1^{\prime}k_2k_1^{\prime\prime}; \\
  p_1&=p_2k_1^{\prime\prime}+k_2k_1^{\prime\prime}p_1^{\prime}; \\
  k_2&=k_2^{\prime}k_1k_2^{\prime\prime}; \\
  p_2&=p_1k_2^{\prime\prime}+k_1k_2^{\prime\prime}p_2^{\prime}.
\end{align*}
By similar arguments as in case (a) we obtain that $k_1=k_2$ and $p_1=p_2$.

In case (g) by statements \eqref{theorem-2.9(5)}, \eqref{theorem-2.9(4)}, and \eqref{theorem-2.9(6)} of Theorem~\ref{theorem-2.9} we have that
\begin{align*}
  \beta_{k_1,p_1}&=\alpha_{k_1^{\prime},p_1^{\prime}}\beta_{k_2,p_2}\beta_{k_1^{\prime\prime},p_1^{\prime\prime}}
   =\alpha_{k_1^{\prime},p_1^{\prime}}\beta_{k_2k_1^{\prime\prime},p_2k_1^{\prime\prime}}
   =\beta_{k_1^{\prime}k_2k_1^{\prime\prime},p_2k_1^{\prime\prime}+k_2k_1^{\prime\prime}p_1^{\prime}}; \\
  \beta_{k_2,p_2}&=\beta_{k_2^{\prime},p_1^{\prime}}\beta_{k_1,p_1}\alpha_{k_2^{\prime\prime},p_1^{\prime\prime}}
   =\beta_{k_2^{\prime},p_2^{\prime}}\beta_{k_1k_2^{\prime\prime},p_1k_2^{\prime\prime}}
   =\beta_{k_2^{\prime}k_1k_2^{\prime\prime},p_2^{\prime}p_1k_2^{\prime\prime}},
\end{align*}
which implies the following equalities
\begin{align*}
  k_1&=k_1^{\prime}k_2k_1^{\prime\prime}; \\
  p_1&=p_2k_1^{\prime\prime}+k_2k_1^{\prime\prime}p_1^{\prime}; \\
  k_2&=k_2^{\prime}k_1k_2^{\prime\prime}; \\
  p_2&=p_2^{\prime}p_1k_2^{\prime\prime}.
\end{align*}
The conditions that  $k_1,k_2,k_1^{\prime},k_2^{\prime},k_1^{\prime\prime},k_2^{\prime\prime}$ are positive integers, $k_2^{\prime}\geqslant 2$, and $k_1^{\prime\prime}\geqslant 2$ contradict the equalities
\begin{equation*}
  k_1=k_1^{\prime}k_2k_1^{\prime\prime}=k_1^{\prime}k_2^{\prime}k_1k_2^{\prime\prime}k_1^{\prime\prime}.
\end{equation*}
Hence case (g) does not hold.

In case (h) by statements \eqref{theorem-2.9(5)} and \eqref{theorem-2.9(4)} of Theorem~\ref{theorem-2.9} we have that
\begin{align*}
  \beta_{k_1,p_1}&=\alpha_{k_1^{\prime},p_1^{\prime}}\beta_{k_2,p_2}\beta_{k_1^{\prime\prime},p_1^{\prime\prime}}
   =\alpha_{k_1^{\prime},p_1^{\prime}}\beta_{k_2k_1^{\prime\prime},p_2k_1^{\prime\prime}}
   =\beta_{k_1^{\prime}k_2k_1^{\prime\prime},p_2k_1^{\prime\prime}+k_2k_1^{\prime\prime}p_1^{\prime}}; \\
  \beta_{k_2,p_2}&=\beta_{k_2^{\prime},p_1^{\prime}}\beta_{k_1,p_1}\beta_{k_2^{\prime\prime},p_1^{\prime\prime}} 
   =\beta_{k_2^{\prime},p_2^{\prime}}\beta_{k_1k_2^{\prime\prime},p_1k_2^{\prime\prime}}
   =\beta_{k_2^{\prime}k_1k_2^{\prime\prime},p_2^{\prime}p_1k_2^{\prime\prime}},
\end{align*}
which implies the following equalities
\begin{align*}
  k_1&=k_1^{\prime}k_2k_1^{\prime\prime}; \\
  p_1&=p_2k_1^{\prime\prime}+k_2k_1^{\prime\prime}p_1^{\prime}; \\
  k_2&=k_2^{\prime}k_1k_2^{\prime\prime}; \\
  p_2&=p_2^{\prime}p_1k_2^{\prime\prime}.
\end{align*}
Similar arguments as in case (g) show that case (h) does not hold.

The proofs in cases (i) and (j) are similar to cases (c) and (g), respectively.

In case (k) by statements \eqref{theorem-2.9(5)}, \eqref{theorem-2.9(4)}, and \eqref{theorem-2.9(6)} of Theorem~\ref{theorem-2.9} we have that
\begin{align*}
  \beta_{k_1,p_1}&=\beta_{k_1^{\prime},p_1^{\prime}}\beta_{k_2,p_2}\alpha_{k_1^{\prime\prime},p_1^{\prime\prime}}
   = \beta_{k_1^{\prime},p_1^{\prime}}\beta_{k_2k_1^{\prime\prime},p_2k_1^{\prime\prime}}
   = \beta_{k_1^{\prime}k_2k_1^{\prime\prime},p_1^{\prime}p_2k_1^{\prime\prime}}; \\
  \beta_{k_2,p_2}&=\beta_{k_2^{\prime},p_1^{\prime}}\beta_{k_1,p_1}\alpha_{k_2^{\prime\prime},p_1^{\prime\prime}}
   = \beta_{k_2^{\prime},p_2^{\prime}}\beta_{k_1k_2^{\prime\prime},p_1k_2^{\prime\prime}}
   = \beta_{k_2^{\prime}k_1k_2^{\prime\prime},p_2^{\prime}p_1k_2^{\prime\prime}},
\end{align*}
which implies the following equalities
\begin{align*}
  k_1&=k_1^{\prime}k_2k_1^{\prime\prime}; \\
  p_1&=p_1^{\prime}p_2k_1^{\prime\prime}; \\
  k_2&=k_2^{\prime}k_1k_2^{\prime\prime}; \\
  p_2&=p_2^{\prime}p_1k_2^{\prime\prime}.
\end{align*}
The conditions that  $k_1,k_2,k_1^{\prime},k_2^{\prime},k_1^{\prime\prime},k_2^{\prime\prime}$ are positive integers, $k_1^{\prime}\geqslant 2$, and $k_1^{\prime}\geqslant 2$ contra\-dict the equalities
\begin{equation*}
  k_1=k_1^{\prime}k_2k_1^{\prime\prime}=k_1^{\prime}k_2^{\prime}k_1k_2^{\prime\prime}k_1^{\prime\prime}.
\end{equation*}
Hence case (k) does not hold.

In case (l) by statements \eqref{theorem-2.9(5)} and \eqref{theorem-2.9(6)} of Theorem~\ref{theorem-2.9} we have that
\begin{align*}
  \beta_{k_1,p_1}&=\beta_{k_1^{\prime},p_1^{\prime}}\beta_{k_2,p_2}\alpha_{k_1^{\prime\prime},p_1^{\prime\prime}}
   =\beta_{k_1^{\prime},p_1^{\prime}}\beta_{k_2k_1^{\prime\prime},p_2k_1^{\prime\prime}}
   =\beta_{k_1^{\prime}k_2k_1^{\prime\prime},p_1^{\prime}p_2k_1^{\prime\prime}}; \\
  \beta_{k_2,p_2}&=\beta_{k_2^{\prime},p_1^{\prime}}\beta_{k_1,p_1}\beta_{k_2^{\prime\prime},p_1^{\prime\prime}}
   =\beta_{k_2^{\prime},p_2^{\prime}}\beta_{k_1k_2^{\prime\prime},p_1k_2^{\prime\prime}}
   =\beta_{k_2^{\prime}k_1k_2^{\prime\prime},p_2^{\prime}p_1k_2^{\prime\prime}},
\end{align*}
which implies the following equalities
\begin{align*}
  k_1&=k_1^{\prime}k_2k_1^{\prime\prime}; \\
  p_1&=p_1^{\prime}p_2k_1^{\prime\prime}; \\
  k_2&=k_2^{\prime}k_1k_2^{\prime\prime}; \\
  p_2&=p_2^{\prime}p_1k_2^{\prime\prime}.
\end{align*}
By similar arguments as in case (k) we get that case (l) does not hold.

The proofs in cases (m) and (n) are similar to cases (d) and (h), respectively.

In case (o) by Theorem~\ref{theorem-2.9}\eqref{theorem-2.9(5)} we have that
\begin{align*}
  \beta_{k_1,p_1}&=\beta_{k_1^{\prime},p_1^{\prime}}\beta_{k_2,p_2}\beta_{k_1^{\prime\prime},p_1^{\prime\prime}}
   =\beta_{k_1^{\prime},p_1^{\prime}}\beta_{k_2k_1^{\prime\prime},p_2k_1^{\prime\prime}}
   =\beta_{k_1^{\prime}k_2k_1^{\prime\prime},p_1^{\prime}p_2k_1^{\prime\prime}}; \\
  \beta_{k_2,p_2}&=\beta_{k_2^{\prime},p_1^{\prime}}\beta_{k_1,p_1}\beta_{k_2^{\prime\prime},p_1^{\prime\prime}} 
   =\beta_{k_2^{\prime},p_2^{\prime}}\beta_{k_1k_2^{\prime\prime},p_1k_2^{\prime\prime}} 
   =\beta_{k_2^{\prime}k_1k_2^{\prime\prime},p_2^{\prime}p_1k_2^{\prime\prime}},
\end{align*}
which implies the following equalities
\begin{align*}
  k_1&=k_1^{\prime}k_2k_1^{\prime\prime}; \\
  p_1&=p_1^{\prime}p_2k_1^{\prime\prime}; \\
  k_2&=k_2^{\prime}k_1k_2^{\prime\prime}; \\
  p_2&=p_2^{\prime}p_1k_2^{\prime\prime}.
\end{align*}
By similar arguments as in case (k) we get that case (o) does not hold.
\end{proof}

Proposition~\ref{proposition-3.3} implies the following theorem.

\begin{theorem}\label{theorem-3.4}
If $\mathscr{F}=\{[0),[1)\}$ then Green's relations $\mathscr{R}$, $\mathscr{L}$, $\mathscr{H}$, $\mathscr{D}$, and $\mathscr{J}$  on the monoid $\operatorname{\boldsymbol{End}}^1_*(\boldsymbol{B}_{\omega}^{\mathscr{F}})$ coincide with the relation of equality.
\end{theorem}

%%%%%%%%%%%%%%%%%%%%%%%%%%%%%%%%%%%%%%%%%%%%%%%%%%%%%
\section*{{Acknowledgements}}

The authors acknowledge the referee for his/her comments and suggestions.

%\newpage

\end{document}